\documentclass[10pt,a4paper]{amsart}
\usepackage[a4paper,marginratio=1:1]{geometry}
\usepackage[initials,non-sorted-cites]{amsrefs}
\usepackage{graphicx}
\usepackage{mybibspec}
\usepackage[shortlabels]{enumitem}
\usepackage[symbol,perpage]{footmisc}
\usepackage{braket,tensor,fouridx,bm}

\numberwithin{equation}{section}

\newcommand{\bdry}{\partial}
\newcommand{\abs}[1]{\lvert#1\rvert}

\newcommand{\norm}[1]{\lVert#1\rVert}
\newcommand{\conj}[1]{\overline{#1}}
\newcommand{\conjsmash}[1]{\smash{\overline{#1}}}
\newcommand{\leftupperstar}[1]{\fourIdx{*}{}{}{}{#1}}
\newcommand{\leftupperL}[1]{\fourIdx{\mathrm{L}}{}{}{}{#1}}

\newcommand{\Nsym}{{N_\mathrm{sym}}}
\newcommand{\Nskew}{{N_\mathrm{skew}}}
\newcommand{\dotNsym}{{\dot{N}_\mathrm{sym}}}
\newcommand{\dotNskew}{{\dot{N}_\mathrm{skew}}}
\newcommand{\Ric}{\mathord{\mathrm{Ric}}}
\renewcommand{\Im}{\operatorname{Im}}
\renewcommand{\Re}{\operatorname{Re}}
\newcommand{\GL}{\mathit{GL}}

\newcommand{\PSU}{\mathit{PSU}}
\newcommand{\transpose}[1]{\fourIdx{t}{}{}{}{#1}}
\DeclareMathOperator{\End}{End}
\DeclareMathOperator{\tr}{tr}

\DeclareMathOperator{\Sym}{Sym}
\DeclareMathOperator{\Diff}{Diff}
\DeclareMathOperator{\Homeo}{Homeo}
\DeclareMathOperator{\dom}{dom}
\DeclareMathOperator{\id}{id}
\DeclareMathOperator{\Fl}{Fl}
\DeclareMathOperator{\dist}{dist}

\theoremstyle{theorem}
\newtheorem{thm}{Theorem}[section]
\newtheorem{prop}[thm]{Proposition}
\newtheorem{lem}[thm]{Lemma}

\theoremstyle{definition}
\newtheorem{dfn}[thm]{Definition}
\theoremstyle{remark}

\newtheorem{rem}[thm]{Remark}

\title[]{Canonical almost complex structures on\\ ACH Einstein manifolds}
\author{Yoshihiko Matsumoto}
\address{Department of Mathematics, Graduate School of Science, Osaka University, Toyonaka, Osaka 560-0043, Japan}
\email{matsumoto@math.sci.osaka-u.ac.jp}

\begin{document}

\maketitle

\begin{abstract}
	On asymptotically complex hyperbolic (ACH) Einstein manifolds,
	we consider a certain variational problem for almost complex structures compatible with the metric,
	for which the linearized Euler--Lagrange equation
	at K\"ahler-Einstein structures is given by the Dolbeault Laplacian
	acting on $(0,1)$-forms with values in the holomorphic tangent bundle.
	A deformation result of Einstein ACH metrics associated with
	critical almost complex structures for this variational problem is given.
	It is also shown that the asymptotic expansion of a critical almost complex structure is
	determined by the induced (possibly non-integrable) CR structure on the boundary at infinity
	up to a certain order.
\end{abstract}

\section{Introduction}

Asymptotically complex hyperbolic (ACH) Einstein spaces,
the complex analog of asymptotically hyperbolic (AH) Einstein spaces,
has been studied by some authors.
The main issue is to describe the interplay between the space itself
and the induced Cauchy--Riemann (CR) structure on the boundary (the \emph{conformal infinity}).
The fundamental problems are such as to determine all the CR structures on the boundary
that are induced by some ACH Einstein metric
and to describe analytic/geometric properties of ACH Einstein spaces
in terms of the conformal infinity.

The AH setting, in which the role of CR structures is instead played by conformal structures,
has been enthusiastically pursued---partly because of physical interest in AdS/CFT correspondence.
By contrast, the ACH setting, which is mathematically one step more intricate,
needs further attention.
Moreover, these two settings can be seen as the first two instances of ``asymptotically symmetric'' spaces
(see Biquard \cites{Biquard-99,Biquard-00} and Biquard--Mazzeo \cite{Biquard-Mazzeo-06});
hence the study of ACH spaces serves as an attempt at a fuller appreciation of this general perspective.

While our subject can be placed in such a context,
it can also be seen as a generalization of the classical studies of complete K\"ahler-Einstein metrics
on bounded strictly pseudoconvex domains in complex manifolds of dimension $n\ge 2$.
Fefferman \cite{Fefferman-76} pioneered the field, and then
the global existence of such metrics (on domains in Stein manifolds) was proved by Cheng and Yau \cite{Cheng-Yau-80}.
Significant applications for those K\"ahler-Einstein metrics are made possible by the fact that
their asymptotic behavior at the boundary can be analyzed fairly well in terms of the CR structure of the boundary,
which was actually the point made in \cite{Fefferman-76}.

In this article, we consider the problem of introducing an almost complex structure to a given
arbitrary ACH Einstein space that extends the CR structure on the boundary in an appropriate sense
and in a canonical manner. Doing so generalizes the Cheng--Yau situation described above.

The significance of this idea is recognized by, for example,
recalling the work of Burns and Epstein \cite{Burns-Epstein-90}.
They studied renormalized integrals of the Chern forms of the
Cheng--Yau complete K\"ahler-Einstein metric,
and they were able to express such an integral as the sum of
a CR invariant of the boundary and a topological term.
(This construction was recently revisited and made more accessible by Marugame \cite{Marugame-16}.)
Because the Chern forms are concerned, it is crucial for this construction
that the domain carries not only a metric but also a complex structure,
which is, in this case, naturally inherited from the ambient complex manifold.
This is the first obstacle to extending the Burns--Epstein construction to general ACH Einstein spaces.

In 4-dimensional spaces,
Biquard and Herzlich \cite{Biquard-Herzlich-05} resolved the issue by constructing
a formal asymptotic expansion of a complex structure with respect to which the metric is asymptotically K\"ahler.
This approach makes sense because any almost
CR structure on 3-dimensional boundary automatically satisfies the formal integrability condition.
In higher dimensions, the conformal infinities of ACH spaces are not necessarily integrable but just
\emph{compatible almost CR structures} adapted to a contact structure
(the definition is given in Section \ref{subsec:compatible-almost-CR}), and because of this, one has to find
another condition on almost complex structures that replaces K\"ahlerness.

To obtain an appropriate condition, we want to use some functional of almost complex structures $J$.
More precisely, we consider those $J$ that are
compatible with a given ACH metric $g$ (in the sense that $g$ is Hermitian with respect to $J$)
and are extensions of the conformal infinity of $g$ (whose meaning is made precise later),
in which case we call the pair $(g,J)$ an \emph{ACH almost Hermitian structure}.
Our functional should be defined in the space of $J$ for which $(g,J)$ is an ACH almost Hermitian structure.

Then, there is one functional that serves our purpose:
\begin{equation}
	\label{eq:Functional-for-almost-complex-structure}
	\mathcal{E}_g[J]=\int_X\left(\abs{N}^2+\frac{1}{2}\abs{\tau}^2\right)dV_g.
\end{equation}
Here, $N$ is the Nijenhuis tensor, and $\tau$ is the trace of $T$,
where $T$ is the $(2,1)$-part of the exterior derivative of the fundamental 2-form
$F(\mathord{\cdot},\mathord{\cdot})=g(J\mathord{\cdot},\mathord{\cdot})$
(see Section \ref{sec:Ehresmann-Libermann-connection} for our normalization).
It should be noted that the right-hand side of \eqref{eq:Functional-for-almost-complex-structure}
diverges in general in our setting and it has to be taken as a formal expression.
However, the associated Euler--Lagrange equation makes sense,
and in terms of the canonical Hermitian connection $\nabla$ on $(X,g,J)$ called the Ehresmann--Libermann connection,
the equation is given by
\begin{equation}
	\label{eq:Euler-Lagrange-for-almost-complex-structure}
	\tensor{S}{_i_j}:=
	i\left((\tensor{\nabla}{^k}+\tensor{\tau}{^k})\tensor{N}{_[_i_j_]_k}
	+\frac{1}{2}\tensor{\nabla}{_[_i}\tensor{\tau}{_j_]}
	+\frac{1}{2}\tensor{N}{_[_i_|_k_l}\tensor{T}{_|_j_]^k^l}
	-\frac{1}{4}\tensor{N}{_k_i_j}\tensor{\tau}{^k}
	+\frac{1}{4}\tensor{T}{^k_i_j}\tensor{\tau}{_k}\right)=0,
\end{equation}
where $i$, $j$, $k$, and $l$ are holomorphic indices and Einstein's summation convention is observed.
It is obvious that $J$ is a critical point of $\mathcal{E}_g$ if $(g,J)$ is K\"ahler,
or more generally, if $N=0$ and $\tau=0$ are satisfied,
in which case $(g,J)$ is called \emph{semi-K\"ahler} by Gauduchon \cite{Gauduchon-84}.

Our choice of the functional makes the linearization $P_S$ of the mapping $J\mapsto S$,
which can be regarded as an operator acting on anti-Hermitian 2-forms
(i.e., 2-forms $A$ satisfying $A(J\mathord{\cdot},J\mathord{\cdot})=-A(\mathord{\cdot},\mathord{\cdot})$)
\begin{equation*}
	P_S\colon\Gamma(X,\mathord{\wedge}^2_\mathrm{aH})\to\Gamma(X,\mathord{\wedge}^2_\mathrm{aH}),
\end{equation*}
a Laplace-type differential operator.
Let us focus on the linearization at K\"ahler-Einstein structures for more specificity.
In this case, if we identify $\Gamma(X,\mathord{\wedge}^2_\mathrm{aH})$
with a subspace of the set of $(0,1)$-forms with values in the holomorphic tangent bundle $T^{1,0}$
by the duality induced by the metric, then $P_S$ is identical to
the Dolbeault Laplacian $\Delta_{\conj{\partial}}$:
if $\Ric(g)=\lambda g$ (where $\lambda=-(n+1)$ for ACH K\"ahler-Einstein metrics), then
\begin{equation}
	\label{eq:linearized-Euler-Lagrange}
	\tensor{(P_SA)}{_i_j}=\frac{1}{2}\tensor{(\Delta_{\conj{\partial}}A)}{_i_j}
	=-\frac{1}{2}(\tensor{\nabla}{_k}\tensor{\nabla}{^k}\tensor{A}{_i_j}+\lambda\tensor{A}{_i_j}).
\end{equation}
This is an important property of our $\mathcal{E}_g$,
without which our construction of solutions of \eqref{eq:system-of-equations} will be
much more complicated, if not impossible.
This property is also preferable from an aesthetic viewpoint because, on closed complex manifolds,
the set of infinitesimal deformations of complex structures is identified with
the space of harmonic $T^{1,0}$-valued $(0,1)$-forms,
or equivalently, with the cohomology group $H^1(X,\Theta)$,
$\Theta$ being the sheaf of germs of holomorphic vector fields (see, e.g., Kodaira \cite{Kodaira-05}).

We now formulate our first result in this paper, which is a perturbative global existence for the system
\begin{equation}
	\label{eq:system-of-equations}
	\Ric(g)=-(n+1)g,\qquad S=0.
\end{equation}

Recall the result of Roth \cite{Roth-99-Thesis}
and Biquard \cite{Biquard-00} (see also the English translation \cite{Biquard-00-English})
on deformations of Einstein ACH metrics:
an Einstein ACH metric $g$ can be deformed into a family of such metrics
parametrized by the conformal infinity when the $L^2$ kernel $\ker_{(2)}P_{\Hat{E}}$
of the linearized gauged Einstein operator
$P_{\Hat{E}}=\nabla_g^*\nabla_g-2\mathring{R}_g$ acting on symmetric 2-tensors vanishes
(here, $\nabla_g$ is the Levi-Civita connection, and $\mathring{R}_g$ is the pointwise linear action of
the curvature tensor).
Our claim in the theorem below is its variation, and is roughly the following:
if what is given in the beginning is not only a metric but an ACH almost Hermitian structure
that is K\"ahler-Einstein (or an \emph{ACH K\"ahler-Einstein structure} for short),
then under the same assumption on $P_{\Hat{E}}$,
one can similarly construct a family of deformed ACH almost Hermitian structures
satisfying \eqref{eq:system-of-equations}.

To state the theorem precisely,
let $\mathcal{C}_H^{2,\alpha}$ be
the set of all almost CR structures of class $C^{2,\alpha}$ compatible with a contact distribution $H$.
For $\delta\in(0,1]$,
the set of ACH metrics (resp.~ACH almost Hermitian structures) of ``class $C^{2,\alpha}_\delta$''
is denoted by $\mathcal{M}_\delta^{2,\alpha}$ (resp.~$\tilde{\mathcal{M}}_\delta^{2,\alpha}$),
whose definition and the notion of smooth families of elements thereof
are discussed in detail in Section \ref{sec:ACH}.
It is always assumed that $n\ge 2$ in the sequel, and $\alpha\in(0,1)$ is arbitrarily fixed.

\begin{thm}
	\label{thm:global-existence}
	Let $\overline{X}$ be a compact smooth manifold-with-boundary of dimension $2n$
	whose boundary $\bdry{X}$ is equipped with a contact distribution $H$.
	Suppose that $(g,J)\in\tilde{\mathcal{M}}_\delta^{2,\alpha}$ is an ACH K\"ahler-Einstein structure
	on the interior $X$ satisfying $\ker_{(2)}P_{\Hat{E}}=0$,
	whose conformal infinity $\gamma_0$ belongs to $\mathcal{C}^{2,\alpha}_H$.
	Then, for a sufficiently small $C^{2,\alpha}$-neighborhood $\mathcal{U}$ of $\gamma_0$
	in $\mathcal{C}^{2,\alpha}_H$,
	there exists a family $(g_\gamma,J_\gamma)$ of elements of $\tilde{\mathcal{M}}_\delta^{2,\alpha}$
	smoothly parametrized by the conformal infinity $\gamma\in\mathcal{U}$
	with the following properties:
	\begin{enumerate}[(i)]
		\item
			$(g_{\gamma_0},J_{\gamma_0})=(g,J)$.
		\item
			$(g_\gamma,J_\gamma)$ satisfies \eqref{eq:system-of-equations} for each $\gamma\in\mathcal{U}$.
	\end{enumerate}
	Moreover, the family can be constructed in such a way that, for each $\gamma$,
	there exists a $C^{2,\alpha}_\delta$-neighborhood $\mathcal{V}$ of $(g_\gamma,J_\gamma)$ in
	$\tilde{\mathcal{M}}_\delta^{2,\alpha}$, such that
	if $(g',J')\in\mathcal{V}$ satisfies \eqref{eq:system-of-equations},
	then there exists $\Phi\in\Diff(X)\cap\Homeo(\overline{X})$ for which $\Phi|_{\bdry X}=\id_{\bdry X}$ and
	$\Phi^*(g',J')=(g_\gamma,J_\gamma)$.
\end{thm}

It is well known that the assumption $\ker_{(2)}P_{\Hat{E}}=0$ is satisfied
when $g$ has negative sectional curvature;
see \cite{Roth-99-Thesis}*{Proposition 4.8} and the comment following \cite{Biquard-00}*{D\'efinition I.1.6}.
Also, the author proved in \cite{Matsumoto-preprint16} that
the Cheng--Yau metric on any smoothly bounded strictly pseudoconvex domain in a Stein manifold of complex dimension
$n\ge 3$ actually satisfies $\ker_{(2)}P_{\Hat{E}}=0$, which provides an abundant amount of
ACH K\"ahler-Einstein spaces
to which Theorem \ref{thm:global-existence} is applicable.

For future applications,
knowing the asymptotic expansion of $(g,J)$ that (approximately) solves \eqref{eq:system-of-equations}
is also important.
This is achieved by our second result below.
Note that the assertion regarding the metric $g$ is contained in
\cites{Matsumoto-14,Matsumoto-13-Thesis}, and our focus here lies on the expansion of $J$.

Let $\mathcal{C}_H$ denote the set of all smooth almost CR structures compatible with $H$.

\begin{thm}
	\label{thm:formal-existence}
	Let $\overline{X}$ be a compact smooth manifold-with-boundary of dimension $2n$
	whose boundary is equipped with a contact distribution $H$.
	Then, for any prescribed conformal infinity $\gamma\in\mathcal{C}_H$,
	in a neighborhood of $\bdry X$
	there exists an ACH almost Hermitian structure $(g,J)$ that is smooth up to the boundary satisfying
	\begin{equation*}
		\Ric(g)=-(n+1)g+O(x^{2n})\qquad\text{and}\qquad
		S=O(x^{2n}),
	\end{equation*}
	where $x$ is any boundary defining function of $\overline{X}$.
	Up to the action of diffeomorphisms near the boundary
	that restricts to the identity on $\bdry{X}$,
	such $(g,J)$ is uniquely determined up to an $O(x^{2n})$ ambiguity,
	in such a way that the local geometry of $\gamma$ determines $(g,J)$ locally.
\end{thm}

The remainder of this article is organized as follows.
Basic facts regarding the Ehresmann--Libermann connection are summarized in the first half of
Section \ref{sec:Ehresmann-Libermann-connection}, and in its second half we discuss
the integration-by-parts formula expressed in terms of this connection
and the variations of the torsion for deformations of almost complex structures.
In Section \ref{sec:Euler-Lagrange-equation}, we explicitly derive
the Euler--Lagrange equation of the functional $\mathcal{E}_g$ and compute its linearization
at K\"ahler-Einstein structures (that is, we verify \eqref{eq:linearized-Euler-Lagrange}).
It is worth noting that there is also a way to obtain \eqref{eq:linearized-Euler-Lagrange}
without writing down the Euler--Lagrange equation itself
(see Remark \ref{rem:linearized-Euler-Lagrange-equation}).
Section \ref{sec:ACH} is devoted to our precise definitions regarding ACH metrics and ACH almost Hermitian structures.
Moreover, we offer here a slightly modified version of the Fredholm theorem
of \cite{Roth-99-Thesis} and \cite{Biquard-00} regarding geometric linear differential operators,
and we calculate the indicial roots of $P_S$.
Then, Theorems \ref{thm:global-existence} and \ref{thm:formal-existence} are
proved in Sections \ref{sec:proof-of-global-existence} and \ref{sec:approximate-solutions}, respectively.
We conclude the article by discussing a partial characterization of our functional
in Section \ref{sec:possible-functionals}.

I am thankful to Rafe Mazzeo and Olivier Biquard for fruitful discussions.
I would also like to express my gratitude to the anonymous reviewer
for careful reading of my manuscript and valuable comments,
based on which I was able to correct several mistakes and to improve the exposition.
Most part of this work was carried out during the author's visit to Stanford University,
which I thank for its warm and helpful working environment.
This work was partially supported by JSPS KAKENHI Grant Number JP17K14189 and JSPS Overseas Research Fellowship.

\section{Ehresmann--Libermann connection}
\label{sec:Ehresmann-Libermann-connection}

The Ehresmann--Libermann connection $\nabla$ on almost Hermitian manifolds
is a natural generalization of the Chern connection on Hermitian manifolds.
It is the unique linear connection that respects the almost Hermitian structure
whose torsion has vanishing $(1,1)$ part.

In order to study this connection, we begin by constructing the Lichnerowicz connection,
another canonical connection on almost Hermitian manifolds.
Then we describe the Ehresmann--Libermann connection $\nabla$ in terms of it.
This makes the relation between $\nabla$ and the Levi-Civita connection clear,
which is useful for deriving the integration-by-parts formula in the latter part of this section.
We will also establish variational formulae of the torsion of $\nabla$.

The main references for this section are
Gauduchon \cite{Gauduchon-97}, Kobayashi \cite{Kobayashi-03}, and
Tosatti--Weinkove--Yau \cite{Tosatti-Weinkove-Yau-08}.

\subsection{Lichnerowicz connection}

Let $(g,J)$ be an almost Hermitian structure on a manifold of dimension $2n$,
that is, a pair of a Riemannian metric $g$ and an almost complex structure $J$ such that
$g(J\mathord{\cdot},J\mathord{\cdot})=g(\mathord{\cdot},\mathord{\cdot})$.
Take the eigendecomposition $T_\mathbb{C}=T^{1,0}\oplus\conj{T^{1,0}}$ of the complexified tangent bundle,
and let $\pi_{1,0}\colon T_\mathbb{C}\to T^{1,0}$ be the natural projection.
The \emph{Lichnerowicz connection} of $(g,J)$ is the Hermitian connection $\leftupperL\nabla$
given by, for any vector field $V$ and any $(1,0)$ vector field $W$,
\begin{equation*}
	\leftupperL\nabla_VW=\pi_{1,0}(\leftupperstar\nabla_VW),
\end{equation*}
where $\leftupperstar\nabla$ is the Levi-Civita connection of $g$.
Note that $\leftupperL\nabla$ is uniquely extended to a connection of $T_\mathbb{C}$ by
claiming that it is a real connection.

We can also express the definition in terms of the connection forms as follows.
Take a local frame $\set{Z_i}$ of $T^{1,0}$,
and set $Z_{\conj{i}}=\conj{Z_i}$ so that $\set{Z_i,Z_{\conj{i}}}$ is a local frame of $T_\mathbb{C}$.
The components of the Levi-Civita connection form $\leftupperstar{\omega}$ with respect to this frame
are classified into four types,
\begin{equation*}
	\tensor{\leftupperstar{\omega}}{_i^j},\qquad
	\tensor{\leftupperstar{\omega}}{_i^{\conj{j}}},\qquad
	\tensor{\leftupperstar{\omega}}{_{\conj{i}}^j},\qquad
	\tensor{\leftupperstar{\omega}}{_{\conj{i}}^{\conj{j}}},
\end{equation*}
satisfying
\begin{equation*}
	\tensor{\leftupperstar{\omega}}{_{\conj{i}}^j}=\conj{\tensor{\leftupperstar{\omega}}{_i^{\conj{j}}}}
	\qquad\text{and}\qquad
	\tensor{\leftupperstar{\omega}}{_{\conj{i}}^{\conj{j}}}=\conj{\tensor{\leftupperstar{\omega}}{_i^j}}.
\end{equation*}
Then, the Lichnerowicz connection form $\tensor{\leftupperL\omega}{_i^j}$ is given by
$\tensor{\leftupperL\omega}{_i^j}=\tensor{\leftupperstar{\omega}}{_i^j}$.

Let $\set{\theta^i}$ be the dual coframe of $\set{Z_i}$.
The first structure equation of the Levi-Civita connection reads
\begin{equation*}
	d\theta^i
	=\theta^j\wedge\tensor{\leftupperstar{\omega}}{_j^i}
	+\theta^{\conj{j}}\wedge\tensor{\leftupperstar{\omega}}{_{\conj{j}}^i}.
\end{equation*}
This implies that the torsion form of the Lichnerowicz connection is given by
\begin{equation}
	\label{eq:Lichnerowicz-torsion-and-Levi-Civita}
	\leftupperL{\Theta}^i
	=\theta^{\conj{j}}\wedge\tensor{\leftupperstar{\omega}}{_{\conj{j}}^i}.
\end{equation}
In particular, $\leftupperL{\Theta}^i$ has no $(2,0)$ component.

We define the Nijenhuis tensor $N$ by setting
\begin{equation}
	\label{eq:definition-of-Nijenhuis-tensor}
	[Z_i,Z_j]=-\tensor{N}{^{\conj{k}}_i_j}Z_{\conj{k}}\mod T^{1,0}.
\end{equation}
Then, we derive the following\footnote{Our convention is such that
$(\alpha\wedge\beta)(V,W)=\alpha(V)\beta(W)-\beta(V)\alpha(W)$ for 1-forms $\alpha$ and $\beta$.}:
\begin{equation}
	\label{eq:Lichnerowicz-torsion-20}
	\leftupperL\Theta^{\conj{k}}(Z_i,Z_j)
	=d\theta^{\conj{k}}(Z_i,Z_j)
	=-\theta^{\conj{k}}([Z_i,Z_j])
	=\tensor{N}{^{\conj{k}}_i_j}.
\end{equation}
We define the tensor $T$ by
\begin{equation}
	\label{eq:Lichnerowicz-torsion-11}
	\leftupperL\Theta^{k}(Z_i,Z_{\conj{j}})=\frac{1}{2}\tensor{T}{_i^k_{\conj{j}}}.
\end{equation}
(The order of the indices looks bizarre, but this will ultimately be a good convention;
see \eqref{eq:Ehresmann-Libermann-torsion-20}.)

We will raise/lower the indices of various tensors using the metric $g$, as
$\tensor{T}{^{\conj{k}}_{\conj{i}}_{\conj{j}}}
=\tensor{g}{^l^{\conj{k}}}\tensor{g}{_m_{\conj{i}}}\tensor{T}{_l^m_{\conj{j}}}$, for example.
Furthermore, any tensor that shows up in this article is real unless otherwise stated.
Hence, for example, $\tensor{T}{^k_i_j}$ is automatically set to be
the complex conjugate of $\tensor{T}{^{\conj{k}}_{\conj{i}}_{\conj{j}}}$.

Obviously, $\tensor{N}{^{\conj{k}}_i_j}$ is skew-symmetric in $i$ and $j$.
This is also the case for $\tensor{T}{^k_i_j}$:
\begin{equation}
	\label{eq:T-skew-symmetry}
	\tensor{T}{^k_i_j}=-\tensor{T}{^k_j_i}.
\end{equation}
In fact, \eqref{eq:Lichnerowicz-torsion-and-Levi-Civita} implies that
$\tensor{\leftupperstar\Gamma}{^k_i_{\conj{j}}}=-\frac{1}{2}\tensor{T}{_i^k_{\conj{j}}}$, where
$\tensor{\leftupperstar\Gamma}{^k_i_{\conj{j}}}$ is the Levi-Civita connection coefficient,
and hence \eqref{eq:T-skew-symmetry} follows from the metric compatibility of the Levi-Civita connection.

Let $a$, $b$, and $c$ be indices running through $\set{1,2,\dots,n,\conj{1},\conj{2},\dots,\conj{n}}$.
We introduce the index notation for the torsion by setting
\begin{equation*}
	\leftupperL\Theta^c=\frac{1}{2}\tensor{\leftupperL\Theta}{^c_a_b}\theta^a\wedge\theta^b
\end{equation*}
and requiring that $\tensor{\leftupperL{\Theta}}{^c_a_b}$ is skew-symmetric in $a$ and $b$.
Hence, $\tensor{\leftupperL{\Theta}}{^k_{\conj{i}}_{\conj{j}}}=\tensor{N}{^k_{\conj{i}}_{\conj{j}}}$,
$\tensor{\leftupperL{\Theta}}{^k_i_{\conj{j}}}=\frac{1}{2}\tensor{T}{_i^k_{\conj{j}}}$,
and $\tensor{\leftupperL{\Theta}}{^k_i_j}=0$.

\subsection{Ehresmann--Libermann connection}

We can construct the Ehresmann--Libermann connection $\nabla$ by adding correction terms
to $\leftupperL\nabla$ as follows
(cf.\ the proof of \cite{Kobayashi-03}*{Theorem 2.1}). If the connection forms of
$\nabla$ and $\leftupperL\nabla$ with respect to $\set{Z_i,Z_{\conj{i}}}$ are written as
\begin{equation*}
	\tensor{\omega}{_i^j}=\tensor{\Gamma}{^j_k_i}\theta^k+\tensor{\Gamma}{^j_{\conj{k}}_i}\theta^{\conj{k}}
	\qquad\text{and}\qquad
	\tensor{\leftupperL\omega}{_i^j}
	=\tensor{\leftupperL\Gamma}{^j_k_i}\theta^k+\tensor{\leftupperL\Gamma}{^j_{\conj{k}}_i}\theta^{\conj{k}},
\end{equation*}
then $\Gamma$ should be set as
\begin{equation}
	\label{eq:EL-connection-construction}
	\tensor{\Gamma}{^j_k_i}
	=\tensor{\leftupperL\Gamma}{^j_k_i}-\frac{1}{2}\tensor{T}{^j_i_k},\qquad
	\tensor{\Gamma}{^j_{\conj{k}}_i}
	=\tensor{\leftupperL\Gamma}{^j_{\conj{k}}_i}+\frac{1}{2}\tensor{T}{_i^j_{\conj{k}}}.
\end{equation}

Let the torsion be expressed as $\Theta^c=\frac{1}{2}\tensor{\Theta}{^c_a_b}\theta^a\wedge\theta^b$, as before.
Then,
\begin{equation*}
	\tensor{\Theta}{^k_i_{\conj{j}}}
	=\tensor{\leftupperL\Theta}{^k_i_{\conj{j}}}-\frac{1}{2}\tensor{T}{_i^k_{\conj{j}}}=0,
\end{equation*}
which is the requirement for the Ehresmann--Libermann connection.
(This computation also shows that $\nabla$ is characterized
among almost Hermitian connections by the fact that $\tensor{\Theta}{^k_i_{\conj{j}}}=0$.)
The compensation is that the $(2,0)$ component of the torsion is generally non-vanishing:
\begin{equation}
	\label{eq:Ehresmann-Libermann-torsion-20}
	\tensor{\Theta}{^k_i_j}
	=\tensor{\leftupperL\Theta}{^k_i_j}-\frac{1}{2}\tensor{T}{^k_j_i}+\frac{1}{2}\tensor{T}{^k_i_j}
	=\tensor{T}{^k_i_j}.
\end{equation}
The $(0,2)$ component of the torsion remains unchanged:
$\tensor{\Theta}{^k_{\conj{i}}_{\conj{j}}}=\tensor{N}{^k_{\conj{i}}_{\conj{j}}}$.

The trace of $T$ is denoted by $\tau$:
\begin{equation*}
	\tensor{\tau}{_i}=\tensor{T}{^j_i_j}.
\end{equation*}
Furthermore, we write $\abs{N}^2=N_{ijk}N^{ijk}$ and $\abs{\tau}^2=\tau_i\tau^i$.
Now the functional \eqref{eq:Functional-for-almost-complex-structure} makes sense.

We remark the following fact regarding the fundamental 2-form (cf.\ \cite[Section 6]{Kobayashi-03}).
The first equality below justifies our explanation of $T$ in the introduction.

\begin{prop}
	Let $F$ be the fundamental 2-form associated with an almost Hermitian structure $(g,J)$,
	i.e., $F(\mathord{\cdot},\mathord{\cdot})=g(J\mathord{\cdot},\mathord{\cdot})$.
	Then,
	\begin{equation}
		\label{eq:exterior-derivative-of-F}
		dF=-i(\tensor{N}{_i_j_k}\theta^i\wedge\theta^j\wedge\theta^k
		-\tensor{T}{_{\conj{i}}_j_k}\theta^{\conj{i}}\wedge\theta^j\wedge\theta^k
		+\tensor{T}{_i_{\conj{j}}_{\conj{k}}}\theta^i\wedge\theta^{\conj{j}}\wedge\theta^{\conj{k}}
		-\tensor{N}{_{\conj{i}}_{\conj{j}}_{\conj{k}}}\theta^{\conj{i}}\wedge\theta^{\conj{j}}\wedge\theta^{\conj{k}})
	\end{equation}
	and
	\begin{equation}
		\label{eq:divergence-of-F}
		d^*F=-i(\tensor{\tau}{_i}\theta^i-\tensor{\tau}{_{\conj{i}}}\theta^{\conj{i}}).
	\end{equation}
\end{prop}

\begin{proof}
	Note that $F=i\tensor{g}{_i_{\conj{j}}}\theta^i\wedge\theta^{\conj{j}}$.
	By the first structure equation $d\theta^i=\theta^j\wedge\tensor{\omega}{_j^i}+\Theta^i$
	and the metric compatibility, we obtain
	\begin{equation*}
		dF=i(d\tensor{g}{_i_{\conj{j}}}\wedge\theta^i\wedge\theta^{\conj{j}}
		+\tensor{g}{_i_{\conj{j}}}d\theta^i\wedge\theta^{\conj{j}}
		-\tensor{g}{_i_{\conj{j}}}\theta^i\wedge d\theta^{\conj{j}})
		=i\tensor{g}{_i_{\conj{j}}}(\Theta^i\wedge\theta^{\conj{j}}-\theta^i\wedge\Theta^{\conj{j}}),
	\end{equation*}
	and \eqref{eq:exterior-derivative-of-F} follows.
	The proof of \eqref{eq:divergence-of-F} is deferred to the next subsection.
\end{proof}

The curvature 2-form
$\tensor{\Omega}{_i^j}=d\tensor{\omega}{_i^j}-\tensor{\omega}{_i^k}\wedge\tensor{\omega}{_k^j}$
of the Ehresmann--Libermann connection will be needed in Section \ref{sec:possible-functionals}.
We express its coefficients as $\tensor{\Omega}{_i^j}=\frac{1}{2}\tensor{R}{_i^j_a_b}\theta^a\wedge\theta^b$,
where $\tensor{R}{_i^j_a_b}$ is skew-symmetric in $a$ and $b$, whence
\begin{equation*}
	\tensor{\Omega}{_i^j}
	=\tensor{R}{_i^j_k_{\conj{l}}}\theta^k\wedge\theta^{\conj{l}}
	+\frac{1}{2}\tensor{R}{_i^j_k_l}\theta^k\wedge\theta^l
	+\frac{1}{2}\tensor{R}{_i^j_{\conj{k}}_{\conj{l}}}\theta^{\conj{k}}\wedge\theta^{\conj{l}}.
\end{equation*}
Note that our convention for $\tensor{R}{_i^j_a_b}$ amounts to saying that
\begin{equation*}
	(\tensor{\nabla}{_{Z_a}}\tensor{\nabla}{_{Z_b}}
	-\tensor{\nabla}{_{Z_b}}\tensor{\nabla}{_{Z_a}}
	-\tensor{\nabla}{_{[Z_a,Z_b]}})Z_i
	=\tensor{R}{_i^j_a_b}Z_j,
\end{equation*}
or in the index notation,
\begin{equation}
	\label{eq:commutation-covariant-derivatives}
	(\tensor{\nabla}{_a}\tensor{\nabla}{_b}-\tensor{\nabla}{_b}\tensor{\nabla}{_a})V^i
	=\tensor{R}{_j^i_a_b}V^j-\tensor{\Theta}{^c_a_b}\tensor{\nabla}{_c}\tensor{V}{^i}.
\end{equation}

We write $\tensor{R}{_i_{\conj{j}}}=\tensor{R}{_k^k_i_{\conj{j}}}$ and $R=\tensor{R}{_i^i}=\tensor{R}{_i^i_j^j}$,
and moreover define $\tensor{R}{_{\conj{i}}_j}=\conj{\tensor{R}{_i_{\conj{j}}}}$ as usual.
Because $\nabla$ is a Hermitian connection, it follows that $\tensor{R}{_{\conj{i}}_j}=\tensor{R}{_j_{\conj{i}}}$.

The first Bianchi identity reads
\begin{subequations}
\begin{align}
	\label{eq:first-Bianchi-30}
	\tensor{R}{_\{_j^i_k_l_\}}
	&=\tensor{\nabla}{_\{_j}\tensor{T}{^i_k_l_\}}+\tensor{T}{^i_p_\{_j}\tensor{T}{^p_k_l_\}},\\
	\label{eq:first-Bianchi-21}
	\tensor{R}{_j^i_k_{\conj{l}}}-\tensor{R}{_k^i_j_{\conj{l}}}
	&=\tensor{\nabla}{_{\conj{l}}}\tensor{T}{^i_j_k}
	+\tensor{N}{^i_{\conj{q}}_{\conj{l}}}\tensor{N}{^{\conj{q}}_j_k},\\
	\label{eq:first-Bianchi-12}
	\tensor{R}{_j^i_{\conj{k}}_{\conj{l}}}
	&=\tensor{\nabla}{_j}\tensor{N}{^i_{\conj{k}}_{\conj{l}}}
	+\tensor{N}{^p_{\conj{k}}_{\conj{l}}}\tensor{T}{^i_p_j},\\
	\label{eq:first-Bianchi-03}
	0
	&=\tensor{\nabla}{_\{_{\conj{j}}}\tensor{N}{^i_{\conj{k}}_{\conj{l}}_\}}
	+\tensor{N}{^i_{\conj{p}}_\{_{\conj{j}}}\tensor{T}{^{\conj{p}}_{\conj{k}}_{\conj{l}}_\}},
\end{align}
\end{subequations}
where $\set{\cdots}$ denotes the cyclic summation
(see \cite{Tosatti-Weinkove-Yau-08}*{Equations (2.8)--(2.11)}; the coefficients are modified
in accordance with our normalization).

\subsection{Integration by parts}

Here, for simplicity, we assume that we are in a setting in which boundary terms do not appear.
The content here is discussed by Streets and Tian \cite{Streets-Tian-11}*{Lemma 10.10}
for Hermitian manifolds (i.e., for integrable almost complex structures).

Suppose that $\alpha$ is a $(1,0)$-form (so $\tensor{\alpha}{_{\conj{i}}}=0$).
Then, by \eqref{eq:EL-connection-construction}
\begin{equation*}
	\tensor{\nabla}{^i}\tensor{\alpha}{_i}
	=\tensor{\leftupperstar\nabla}{^i}\tensor{\alpha}{_i}-\frac{1}{2}\tensor{T}{_i^j^i}\tensor{\alpha}{_j}
	=\tensor{\leftupperstar\nabla}{^i}\tensor{\alpha}{_i}-\frac{1}{2}\tensor{\tau}{^i}\tensor{\alpha}{_i},
\end{equation*}
where $\leftupperstar\nabla$ is the Levi-Civita connection.
Moreover, since $\tensor{\leftupperstar\Gamma}{^k_i_{\conj{j}}}=-\frac{1}{2}\tensor{T}{_i^k_{\conj{j}}}$ by
\eqref{eq:Lichnerowicz-torsion-and-Levi-Civita},
\begin{equation*}
	\tensor{\leftupperstar\nabla}{^{\conj{i}}}\tensor{\alpha}{_{\conj{i}}}
	=\tensor{\nabla}{^{\conj{i}}}\tensor{\alpha}{_{\conj{i}}}
	+\frac{1}{2}\tensor{T}{^{\conj{i}}^j_{\conj{i}}}\tensor{\alpha}{_j}
	=\frac{1}{2}\tensor{\tau}{^i}\tensor{\alpha}{_i}.
\end{equation*}
Therefore, the Levi-Civita divergence of $\alpha$ equals
$-(\tensor{\nabla}{^i}\tensor{\alpha}{_i}+\tensor{\tau}{^i}\tensor{\alpha}{_i})$, and hence
\begin{equation}
	\label{eq:divergence-formula}
	\int(\tensor{\nabla}{^i}+\tensor{\tau}{^i})\tensor{\alpha}{_i}dV_g=0.
\end{equation}

We can use this formula in various ways. The simplest application is the following:
if $\alpha$ is a 1-form and $f$ is a function, then by applying \eqref{eq:divergence-formula} to
$f\tensor{\alpha}{_i}$, we obtain
\begin{equation*}
	\int\tensor{\alpha}{^i}\tensor{\nabla}{_i}f\,dV_g
	=-\int((\tensor{\nabla}{_i}+\tensor{\tau}{_i})\tensor{\alpha}{^i})f\,dV_g.
\end{equation*}
Likewise, if $\beta$ is a 2-tensor and $\alpha$ is a 1-form, then
\begin{equation*}
	\int\tensor{\beta}{^i^j}\tensor{\nabla}{_i}\tensor{\alpha}{_j}\,dV_g
	=-\int((\tensor{\nabla}{_i}+\tensor{\tau}{_i})\tensor{\beta}{^i^j})\tensor{\alpha}{_j}\,dV_g
\end{equation*}
(the index $j$ can also be replaced with $\conj{j}$).
Similar operations are applicable to higher-rank tensors as well.

We can now show \eqref{eq:divergence-of-F}. Let us compute
the divergence of a real $(1,1)$-form $\beta=\tensor{\beta}{_i_{\conj{j}}}\theta^i\wedge\theta^{\conj{j}}$
in general.
If $\alpha=\tensor{\alpha}{_a}\theta^a$ is a real 1-form,
then because the torsion of $\nabla$ has vanishing $(1,1)$-part, the $(1,1)$-part of $d\alpha$ is given by
\begin{equation*}
	(d\alpha)^{(1,1)}
	=(\tensor{\nabla}{_i}\tensor{\alpha}{_{\conj{j}}}-\tensor{\nabla}{_{\conj{j}}}\tensor{\alpha}{_i})
	\theta^i\wedge\theta^{\conj{j}}.
\end{equation*}
Therefore,
\begin{equation*}
	(\alpha,d^*\beta)=(d\alpha,\beta)
	=\int(\tensor{\nabla}{^{\conj{j}}}\tensor{\alpha}{^i}-\tensor{\nabla}{^i}\tensor{\alpha}{^{\conj{j}}})
	\tensor{\beta}{_i_{\conj{j}}}\,dV_g
	=\int(-((\tensor{\nabla}{^{\conj{j}}}+\tensor{\tau}{^{\conj{j}}})
	\tensor{\beta}{_i_{\conj{j}}})\tensor{\alpha}{^i}
	+((\tensor{\nabla}{^i}+\tensor{\tau}{^i})
	\tensor{\beta}{_i_{\conj{j}}})\tensor{\alpha}{^{\conj{j}}})dV_g.
\end{equation*}
This implies
$\tensor{(d^*\beta)}{_i}=-(\tensor{\nabla}{^{\conj{j}}}+\tensor{\tau}{^{\conj{j}}})\tensor{\beta}{_i_{\conj{j}}}$.
We obtain \eqref{eq:divergence-of-F} as a special case.

\subsection{Variations}

Suppose that $J_t$ is a smooth one-parameter family of almost complex structures compatible with
a Riemannian metric $g$.
We write $J=J_0$ and $\dot{J}=(dJ_t/dt)|_{t=0}$.
By differentiating $J_t^2=-1$, we see that $\dot{J}$ is an
anti-Hermitian section of the endomorphism bundle $\End(T)$,
and the metric compatibility of $J_t$ implies that $g(\dot{J}\mathord{\cdot},\mathord{\cdot})$ is skew-symmetric.
We set
\begin{equation}
	\label{eq:first-variation-of-J}
	g(\dot{J}\mathord{\cdot},\mathord{\cdot})=A(\mathord{\cdot},\mathord{\cdot}).
\end{equation}
Then $A$ is an anti-Hermitian 2-form.
Using a local $(1,0)$ coframe $\set{\theta^i}$ and its complex conjugate $\set{\theta^{\conj{i}}}$,
we can write
\begin{equation*}
	A=\frac{1}{2}\tensor{A}{_i_j}\theta^i\wedge\theta^j
	+\frac{1}{2}\tensor{A}{_{\conj{i}}_{\conj{j}}}\theta^{\conj{i}}\wedge\theta^{\conj{j}},
\end{equation*}
where $\tensor{A}{_{\conj{i}}_{\conj{j}}}=\conj{\tensor{A}{_i_j}}$, and
$\tensor{A}{_i_j}$ is skew-symmetric in $i$ and $j$.
Then, \eqref{eq:first-variation-of-J} is expressed as
$\tensor{\dot{J}}{_i^{\conj{k}}}\tensor{g}{_j_{\conj{k}}}=\tensor{A}{_i_j}$,
or simply as $\tensor{\dot{J}}{_i_j}=\tensor{A}{_i_j}$.

Let $\nabla_t$ be the Ehresmann--Libermann connection of $(g,J_t)$.
As an intermediate step toward the variational formulae of the torsion,
we express the derivatives of the connection coefficients $\tensor{\Gamma}{^c_a_b}$ of $\nabla_t$
in terms of $\tensor{A}{_i_j}$.
In the computation that follows, $\tensor{\Gamma}{^c_a_b}$ will be the connection coefficients of $\nabla_t$
\emph{with respect to a fixed local frame $\set{Z_i,Z_{\conj{i}}}$
and the dual coframe $\set{\theta^i,\theta^{\conj{i}}}$},
where $Z_i$ are $(1,0)$ vector fields with respect to the original almost complex structure $J$
(and $Z_{\conj{i}}=\conj{Z_i}$).
We also remark that $\tensor{\dot{\Gamma}}{_a_b_c}$, which appears below,
can be understood either as $\tensor{g}{_a_d}\tensor{\dot{\Gamma}}{^d_b_c}$ or
as the derivative of $\tensor{g}{_a_d}\tensor{\Gamma}{^d_b_c}$, because $g$ is independent of $t$.

We begin with generalities that apply to all almost Hermitian connections.
It follows from the metric compatibility that
$\tensor{\dot{\Gamma}}{_c_a_b}+\tensor{\dot{\Gamma}}{_b_a_c}=0$.
The compatibility with almost complex structures implies
$\tensor{\nabla}{_c}\tensor{\dot{J}}{_a^b}-\tensor{\dot{\Gamma}}{^d_c_a}\tensor{J}{_d^b}
+\tensor{\dot{\Gamma}}{^b_c_d}\tensor{J}{_a^d}=0$.
Therefore, we get
\begin{equation*}
	\tensor{\nabla}{_i}\tensor{A}{_j^{\conj{k}}}+2i\tensor{\dot{\Gamma}}{^{\conj{k}}_i_j}=0
	\qquad\text{and}\qquad
	\tensor{\nabla}{_{\conj{i}}}\tensor{A}{_j^{\conj{k}}}+2i\tensor{\dot{\Gamma}}{^{\conj{k}}_{\conj{i}}_j}=0,
\end{equation*}
and hence,
\begin{subequations}
\begin{equation}
	\label{eq:complex-compatibility}
	\tensor{\dot{\Gamma}}{^{\conj{k}}_i_j}
	=\frac{i}{2}\tensor{\nabla}{_i}\tensor{A}{_j^{\conj{k}}},\qquad
	\tensor{\dot{\Gamma}}{^{\conj{k}}_{\conj{i}}_j}
	=\frac{i}{2}\tensor{\nabla}{_{\conj{i}}}\tensor{A}{_j^{\conj{k}}}.
\end{equation}

Next, we use the definition of the Ehresmann--Libermann connection.
Its torsion $\Theta$ has vanishing $(1,1)$ component, which means
$\tensor{\Theta}{^c_a_b}+\tensor{J}{_a^d}\tensor{J}{_b^e}\tensor{\Theta}{^c_d_e}=0$.
Since $\tensor{\dot{\Theta}}{^c_a_b}=2\tensor{\dot{\Gamma}}{^c_[_a_b_]}
=\tensor{\dot{\Gamma}}{^c_a_b}-\tensor{\dot{\Gamma}}{^c_b_a}$, this implies
\begin{equation*}
	\tensor{\dot{\Gamma}}{^c_[_a_b_]}
	+\tensor{J}{_a^d}\tensor{J}{_b^e}\tensor{\dot{\Gamma}}{^c_[_d_e_]}
	+\frac{1}{2}\tensor{\dot{J}}{_a^d}\tensor{J}{_b^e}\tensor{\Theta}{^c_d_e}
	+\frac{1}{2}\tensor{J}{_a^d}\tensor{\dot{J}}{_b^e}\tensor{\Theta}{^c_d_e}=0.
\end{equation*}
Consequently, we obtain
\begin{equation*}
	\tensor{\dot{\Gamma}}{^k_i_{\conj{j}}}-\tensor{\dot{\Gamma}}{^k_{\conj{j}}_i}
	+\frac{i}{2}\tensor{N}{^k_{\conj{j}}^l}\tensor{A}{_i_l}
	+\frac{i}{2}\tensor{T}{^k_i^{\conj{l}}}\tensor{A}{_{\conj{j}}_{\conj{l}}}=0
\end{equation*}
and hence, by the second equality of \eqref{eq:complex-compatibility},
\begin{equation}
	\tensor{\dot{\Gamma}}{^k_{\conj{i}}_j}
	=\frac{i}{2}(\tensor{\nabla}{_j}\tensor{A}{_{\conj{i}}^k}
	+\tensor{N}{^k_{\conj{i}}^l}\tensor{A}{_j_l}
	+\tensor{T}{^k_j^{\conj{l}}}\tensor{A}{_{\conj{i}}_{\conj{l}}}).
\end{equation}
Then, since $\tensor{\dot{\Gamma}}{_{\conj{k}}_i_j}=-\tensor{\dot{\Gamma}}{_j_i_{\conj{k}}}$, we also obtain
\begin{equation}
	\tensor{\dot{\Gamma}}{^k_i_j}
	=-\frac{i}{2}(\tensor{\nabla}{^k}\tensor{A}{_i_j}
	-\tensor{N}{_j_i^{\conj{l}}}\tensor{A}{^k_{\conj{l}}}
	-\tensor{T}{_j^k^l}\tensor{A}{_i_l}).
\end{equation}
\end{subequations}
Thus we have obtained the complete formula of $\tensor{\dot{\Gamma}}{^c_a_b}$.

We now turn to the torsion.
Recall once again that $\tensor{\dot{\Theta}}{^c_a_b}=2\tensor{\dot{\Gamma}}{^c_[_a_b_]}$.
Since
$\tensor{N}{^c_a_b}=\frac{1}{2}(\tensor{\Theta}{^c_a_b}+\tensor{J}{_b^d}\tensor{J}{_e^c}\,\tensor{\Theta}{^e_a_d})$
and
$\tensor{T}{^c_a_b}=\frac{1}{2}(\tensor{\Theta}{^c_a_b}-\tensor{J}{_b^d}\tensor{J}{_e^c}\,\tensor{\Theta}{^e_a_d})$,
by a straightforward computation we obtain
\begin{subequations}
	\label{eq:variation-Nijenhuis}
\begin{align}
	\tensor{\dot{N}}{^k_i_j}
	&=-\frac{i}{2}\tensor{N}{^{\conj{l}}_i_j}\tensor{A}{^k_{\conj{l}}},\\
	\tensor{\dot{N}}{^k_i_{\conj{j}}}
	&=-\frac{i}{2}\tensor{N}{^k_{\conj{j}}^l}\tensor{A}{_i_l},\\
	\tensor{\dot{N}}{^k_{\conj{i}}_{\conj{j}}}
	&=-i\left(\tensor{\nabla}{_[_{\conj{i}}}\tensor{A}{_{\conj{j}}_]^k}
	-\frac{1}{2}\tensor{T}{^{\conj{l}}_{\conj{i}}_{\conj{j}}}\tensor{A}{^k_{\conj{l}}}\right)
\end{align}
\end{subequations}
and
\begin{subequations}
	\label{eq:variation-torsion}
\begin{align}
	\label{eq:variation-torsion-20}
	\tensor{\dot{T}}{^k_i_j}
	&=
	-i\left(\tensor{\nabla}{^k}\tensor{A}{_i_j}
	+\tensor{N}{_[_i_j_]^{\conj{l}}}\tensor{A}{^k_{\conj{l}}}
	+\tensor{T}{_[_i^k^l}\tensor{A}{_j_]_l}
	-\frac{1}{2}\tensor{N}{^{\conj{l}}_i_j}\tensor{A}{^k_{\conj{l}}}\right),\\
	\tensor{\dot{T}}{^k_i_{\conj{j}}}
	&=-\frac{i}{2}\tensor{T}{^k_i_l}\tensor{A}{_{\conj{j}}^l},\\
	\tensor{\dot{T}}{^k_{\conj{i}}_{\conj{j}}}
	&=-\frac{i}{2}\tensor{T}{^{\conj{l}}_{\conj{i}}_{\conj{j}}}\tensor{A}{^k_{\conj{l}}}.
\end{align}
\end{subequations}
Furthermore, it follows from \eqref{eq:variation-torsion-20} that
\begin{equation}
	\label{eq:variation-tau}
	\tensor{\dot{\tau}}{_i}=
	-i\left(\tensor{\nabla}{^j}\tensor{A}{_i_j}
	+\frac{1}{2}\tensor{N}{_i^{\conj{j}}^{\conj{k}}}\tensor{A}{_{\conj{j}}_{\conj{k}}}
	+\frac{1}{2}\tensor{T}{_i^j^k}\tensor{A}{_j_k}
	+\frac{1}{2}\tensor{\tau}{^j}\tensor{A}{_i_j}\right).
\end{equation}

\section{The functional and the Euler--Lagrange equation}
\label{sec:Euler-Lagrange-equation}

Suppose a Riemannian metric $g$ is fixed,
and consider the set $\mathcal{J}_g$ of compatible almost complex structures.
In this section, we first assume that our space is a closed manifold, and we define the functional
$\mathcal{E}_g$ on $\mathcal{J}_g$ by \eqref{eq:Functional-for-almost-complex-structure}; that is,
\begin{equation*}
	\mathcal{E}_g=\mathcal{E}^N_g+\frac{1}{2}\mathcal{E}^\tau_g,
\end{equation*}
where
\begin{equation*}
	\mathcal{E}_g^N[J]=\int\abs{N}^2dV_g
	\qquad\text{and}\qquad
	\mathcal{E}_g^\tau[J]=\int\abs{\tau}^2dV_g.
\end{equation*}
We shall compute the Euler--Lagrange equation of $\mathcal{E}_g$.
Then, the equation itself also makes sense on noncompact manifolds
(or, on noncompact manifolds, we can interpret this as we are considering the relative values of the functional
under compactly supported variations).

Let $J_t$ be a one-parameter smooth family of elements of $\mathcal{J}_g$
and define the tensor $A$ by \eqref{eq:first-variation-of-J}. We write
\begin{equation*}
	\left.\frac{d}{dt}\mathcal{E}_g[J_t]\right|_{t=0}
	=\int((\dot{\mathcal{E}}_g)^{ij}\tensor{A}{_i_j}
	+(\dot{\mathcal{E}}_g)^{\conj{i}\conj{j}}\tensor{A}{_{\conj{i}}_{\conj{j}}})dV_g
	=\int 2\Re((\dot{\mathcal{E}}_g)^{ij}\tensor{A}{_i_j})dV_g,
\end{equation*}
where $\dot{\mathcal{E}}_g$ is skew-symmetric,
and we introduce $\dot{\mathcal{E}}^N_g$ and $\dot{\mathcal{E}}^\tau_g$ similarly.
The symbol $g$ will be omitted from the notation when there is no fear of confusion.

\begin{prop}
	\label{prop:termwise-Euler-Lagrange}
	Under the notation above,
	\begin{align}
		\label{eq:Euler-Lagrange-Nijenhuis}
		\dot{\mathcal{E}}^N_{ij}
			&=i\left((\tensor{\nabla}{^k}+\tensor{\tau}{^k})\tensor{N}{_[_i_j_]_k}
			+\frac{1}{2}\tensor{N}{_[_i_|_k_l}\tensor{T}{_|_j_]^k^l}\right),\\
		\label{eq:Euler-Lagrange-tau}
		\dot{\mathcal{E}}^\tau_{ij}
			&=i\left(\tensor{\nabla}{_[_i}\tensor{\tau}{_j_]}
			-\frac{1}{2}\tensor{N}{_k_i_j}\tensor{\tau}{^k}
			+\frac{1}{2}\tensor{T}{^k_i_j}\tensor{\tau}{_k}\right).
	\end{align}
	Thus $\tensor{\dot{\mathcal{E}}}{_i_j}=\dot{\mathcal{E}}^N_{ij}+\frac{1}{2}\dot{\mathcal{E}}^\tau_{ij}$
	equals $\tensor{S}{_i_j}$ in equation \eqref{eq:Euler-Lagrange-for-almost-complex-structure}.
\end{prop}

\begin{proof}
	These are consequences of \eqref{eq:variation-Nijenhuis} and \eqref{eq:variation-tau}. First,
	\begin{equation*}
		\left.\frac{d}{dt}\mathcal{E}^N[J_t]\right|_{t=0}
		=\int 2\Re(\tensor{N}{^k^i^j}\tensor{\dot{N}}{_k_i_j})dV_g
		=\int\Im(\tensor{N}{^k^i^j}(-\tensor{\nabla}{_i}\tensor{A}{_j_k}+\tensor{\nabla}{_j}\tensor{A}{_i_k}
		+\tensor{T}{^l_i_j}\tensor{A}{_k_l}))dV_g,
	\end{equation*}
	and, using \eqref{eq:divergence-formula}, we obtain
	\begin{equation*}
		\begin{split}
			\left.\frac{d}{dt}\mathcal{E}^N[J_t]\right|_{t=0}
			&=\int\Im(((\tensor{\nabla}{_i}+\tensor{\tau}{_i})\tensor{N}{^k^i^j})\tensor{A}{_j_k}
			-((\tensor{\nabla}{_j}+\tensor{\tau}{_j})\tensor{N}{^k^i^j})\tensor{A}{_i_k}
			+\tensor{N}{^k^i^j}\tensor{T}{^l_i_j}\tensor{A}{_k_l})dV_g\\
			&=\int 2\Im\left(((\tensor{\nabla}{_k}+\tensor{\tau}{_k})\tensor{N}{^i^j^k})\tensor{A}{_i_j}
			+\frac{1}{2}\tensor{N}{^i^k^l}\tensor{T}{^j_k_l}\tensor{A}{_i_j}\right)dV_g.
		\end{split}
	\end{equation*}
	That is,
	\begin{equation*}
		\tensor{(\dot{\mathcal{E}}^N)}{^i^j}
		=-i\left((\tensor{\nabla}{_k}+\tensor{\tau}{_k})\tensor{N}{^[^i^j^]^k}
		+\frac{1}{2}\tensor{N}{^[^i^|^k^l}\tensor{T}{^|^j^]_k_l}\right).
	\end{equation*}
	Then, we obtain \eqref{eq:Euler-Lagrange-Nijenhuis} by taking the complex conjugate.
	Similarly,
	\begin{equation*}
		\begin{split}
			\left.\frac{d}{dt}\mathcal{E}^\tau[J_t]\right|_{t=0}
			&=\int 2\Re(\tensor{\tau}{^i}\tensor{\dot{\tau}}{_i})dV_g\\
			&=\int \Re(\tensor{\tau}{^i}(-2i\tensor{\nabla}{^j}\tensor{A}{_i_j}
			-i\tensor{N}{_i^{\conj{j}}^{\conj{k}}}\tensor{A}{_{\conj{j}}_{\conj{k}}}
			-i\tensor{T}{_i^j^k}\tensor{A}{_j_k}
			-i\tensor{\tau}{^j}\tensor{A}{_i_j}))dV_g\\
			&=\int\Im(2\tensor{\tau}{^i}\tensor{\nabla}{^j}\tensor{A}{_i_j}
			-\tensor{\tau}{^{\conj{k}}}\tensor{N}{_{\conj{k}}^i^j}\tensor{A}{_i_j}
			+\tensor{\tau}{^k}\tensor{T}{_k^i^j}\tensor{A}{_i_j}
			+\tensor{\tau}{^i}\tensor{\tau}{^j}\tensor{A}{_i_j})dV_g\\
			&=\int\Im(-2((\tensor{\nabla}{^j}+\tensor{\tau}{^j})\tensor{\tau}{^i})\tensor{A}{_i_j}
			-\tensor{N}{_{\conj{k}}^i^j}\tensor{\tau}{^{\conj{k}}}\tensor{A}{_i_j}
			+\tensor{T}{_k^i^j}\tensor{\tau}{^k}\tensor{A}{_i_j}
			+\tensor{\tau}{^i}\tensor{\tau}{^j}\tensor{A}{_i_j})dV_g\\
			&=\int\Im(2(\tensor{\nabla}{^i}\tensor{\tau}{^j})\tensor{A}{_i_j}
			-\tensor{N}{_{\conj{k}}^i^j}\tensor{\tau}{^{\conj{k}}}\tensor{A}{_i_j}
			+\tensor{T}{_k^i^j}\tensor{\tau}{^k}\tensor{A}{_i_j})dV_g,
		\end{split}
	\end{equation*}
	where the last equality is because of the skew-symmetry of $\tensor{A}{_i_j}$. Hence
	\begin{equation*}
		\tensor{(\dot{\mathcal{E}}^\tau)}{^i^j}
		=-i\left(\tensor{\nabla}{^[^i}\tensor{\tau}{^j^]}
		-\frac{1}{2}\tensor{N}{_{\conj{k}}^i^j}\tensor{\tau}{^{\conj{k}}}
		+\frac{1}{2}\tensor{T}{_k^i^j}\tensor{\tau}{^k}\right),
	\end{equation*}
	and this implies \eqref{eq:Euler-Lagrange-tau}.
\end{proof}

Next, we compute the linearization of the tensor $S$ with respect to $J$.
Because of our formulation of Theorem \ref{thm:global-existence},
we are exclusively concerned with the linearization at K\"ahler structures,
for which $N=0$ and $T=0$.
(Note that, in this case, the Levi-Civita, Lichnerowicz, and Ehresmann--Libermann connections coincide.)
By \eqref{eq:variation-Nijenhuis} and \eqref{eq:variation-tau},
\begin{align*}
	\ddot{\mathcal{E}}^N_{ij}
	&=i\tensor{\nabla}{^k}\tensor{\dot{N}}{_[_i_j_]_k}
	=-\frac{1}{2}\tensor{\nabla}{^k}
	(\tensor{\nabla}{_[_j}\tensor{A}{_k_]_i}-\tensor{\nabla}{_[_i}\tensor{A}{_k_]_j})
	=-\frac{1}{2}(\tensor{\nabla}{_k}\tensor{\nabla}{^k}\tensor{A}{_i_j}
	-\tensor{R}{_[_i^k}\tensor{A}{_j_]_k}
	+\tensor{\nabla}{_[_i}\tensor{\nabla}{^k}\tensor{A}{_j_]_k}),\\
	\ddot{\mathcal{E}}^\tau_{ij}
	&=i\tensor{\nabla}{_[_i}\tensor{\dot{\tau}}{_j_]}
	=\tensor{\nabla}{_[_i}\tensor{\nabla}{^k}\tensor{A}{_j_]_k}
\end{align*}
and therefore,
\begin{equation}
	\label{eq:linearized-Euler-Lagrange-for-Kahler}
	\dot{S}_{ij}=-\frac{1}{2}(\tensor{\nabla}{_k}\tensor{\nabla}{^k}\tensor{A}{_i_j}
	-\tensor{R}{_[_i^k}\tensor{A}{_j_]_k}).
\end{equation}

The operator $P_S\colon A\mapsto\dot{S}$ has a close connection to the Dolbeault Laplacian
$\Delta_{\conj{\partial}}=\conjsmash{\partial}^*\conj{\partial}+\conj{\partial}\,\conjsmash{\partial}^*$
acting on $(0,1)$-forms with values in the holomorphic tangent bundle $T^{1,0}$.
If we identify the anti-Hermitian 2-form $A$ with $\tensor{A}{_{\conj{i}}^j}\theta^{\conj{i}}\otimes Z_j$,
then
\begin{align*}
	\tensor{(\conjsmash{\partial}^*\conj{\partial}A)}{_{\conj{i}}^j}
	&=-2\tensor{\nabla}{^{\conj{k}}}\tensor{\nabla}{_[_{\conj{k}}}\tensor{A}{_{\conj{i}}_]^j}
	=-\tensor{\nabla}{^{\conj{k}}}\tensor{\nabla}{_{\conj{k}}}\tensor{A}{_{\conj{i}}^j}
	+\tensor{\nabla}{^{\conj{k}}}\tensor{\nabla}{_{\conj{i}}}\tensor{A}{_{\conj{k}}^j},\\
	\tensor{(\conj{\partial}\,\conjsmash{\partial}^*A)}{_{\conj{i}}^j}
	&=-\tensor{\nabla}{_{\conj{i}}}\tensor{\nabla}{^{\conj{k}}}\tensor{A}{_{\conj{k}}^j},
\end{align*}
and hence
\begin{equation}
	\label{eq:Dolbeault-Laplacian}
	\tensor{(\Delta_{\conj{\partial}}A)}{_{\conj{i}}^j}
	=-(\tensor{\nabla}{_{\conj{k}}}\tensor{\nabla}{^{\conj{k}}}\tensor{A}{_{\conj{i}}^j}
	+\tensor{R}{_k^j}\tensor{A}{_{\conj{i}}^k}).
\end{equation}
This means that $\dot{S}$ is (if also regarded as a $T^{1,0}$-valued $(0,1)$-form)
half of the skew-symmetric part of $\Delta_{\conj{\partial}}A$.
In particular, we obtain \eqref{eq:linearized-Euler-Lagrange}, which we reproduce below.

\begin{prop}
	\label{prop:linearized-Euler-Lagrange-for-Kahler-Einstein}
	If $(g,J)$ is K\"ahler-Einstein, with $\Ric(g)=\lambda g$, then the operator
	\begin{equation}
		\label{eq:linearized-Euler-Lagrange-definition}
		P_S\colon\Gamma(X,\mathord{\wedge}^2_\mathrm{aH})\to\Gamma(X,\mathord{\wedge}^2_\mathrm{aH}),\qquad
		A=\dot{J}\mapsto\dot{S}
	\end{equation}
	is given by
	\begin{equation}
		\label{eq:linearized-Euler-Lagrange-bis}
		\tensor{(P_SA)}{_i_j}
		=-\frac{1}{2}(\tensor{\nabla}{_k}\tensor{\nabla}{^k}\tensor{A}{_i_j}+\lambda\tensor{A}{_i_j}).
	\end{equation}
	If $A$ and $P_SA$ are regarded as $T^{1,0}$-valued $(0,1)$-forms, then
	\begin{equation}
		\label{eq:equality-of-linearized-euler-lagrange-and-dolbeault-laplacian}
		P_SA=\frac{1}{2}\Delta_{\conj{\partial}}A.
	\end{equation}
\end{prop}

\begin{rem}
	\label{rem:linearized-Euler-Lagrange-equation}
	The claim $\dot{S}=\frac{1}{2}(\Delta_{\conj{\partial}}A)_{\mathrm{skew}}$
	for K\"ahler structures has the following alternative proof,
	which does not depend on the explicit formula \eqref{eq:Euler-Lagrange-for-almost-complex-structure} of $S$.
	Note that, if $A$, $N$, $\tau$ are understood as
	the $(0,1)$-, $(0,2)$-, $(0,0)$-forms with values in $T^{1,0}$ given by
	\begin{equation*}
		\tensor{A}{_{\conj{i}}^k}\theta^{\conj{i}}\otimes Z_k,\qquad
		\frac{1}{2}\tensor{N}{^k_{\conj{i}}_{\conj{j}}}\theta^{\conj{i}}\wedge\theta^{\conj{j}}\otimes Z_k,\qquad
		\tensor{\tau}{^k}Z_k,
	\end{equation*}
	respectively,
	then at K\"ahler structures, \eqref{eq:variation-Nijenhuis} and \eqref{eq:variation-tau} may be written as
	\begin{equation*}
		\dot{N}=-\frac{i}{2}\,\conj{\partial}A,\qquad
		\dot{\tau}=-i\,\conjsmash{\partial}^*A.
	\end{equation*}
	This implies that
	$N_t=-\frac{i}{2}t\,\conj{\partial}A+O(t^2)$ and $\tau_t=-it\,\conjsmash{\partial}^*A+O(t^2)$, and hence
	\begin{equation*}
		\mathcal{E}[J_t]
		=\frac{1}{2}t^2((\conj{\partial}A,\conj{\partial}A)
		+(\conjsmash{\partial}^*A,\conjsmash{\partial}^*A))+O(t^3)
		=\frac{1}{2}t^2(\Delta_{\conj{\partial}}A,A)+O(t^3).
	\end{equation*}
	Consequently,
	\begin{equation*}
		\frac{d}{dt}\mathcal{E}[J_t]=t(\Delta_{\conj{\partial}}A,A)+O(t^2)
		=t\int\tensor{(\Delta_{\conj{\partial}}A)}{_{\conj{i}}^j}\tensor{A}{^{\conj{i}}_j}dV_g+O(t^2)
		=t\int\Re(\tensor{(\Delta_{\conj{\partial}}A)}{_{\conj{i}}^j}\tensor{A}{^{\conj{i}}_j})dV_g+O(t^2).
	\end{equation*}
	This implies that $\dot{S}=\frac{1}{2}(\Delta_{\conj{\partial}}A)_{\mathrm{skew}}$.
\end{rem}

\section{ACH almost Hermitian structures}
\label{sec:ACH}

In this section, we first describe our basic definitions regarding ACH metrics,
and then we define ACH almost Hermitian structures.
They are followed by the Fredholm theorem for geometric differential operators,
which is a modification of the one considered by
Roth \cite{Roth-99-Thesis} and Biquard \cite{Biquard-00}.
Lee \cite{Lee-06} gives another useful reference on this matter in the AH setting,
which is also referred to to discuss the details here.
Finally, we compute the indicial roots of the operator $P_S$.
This is a requisite to applying the Fredholm theorem in the next section.

\subsection{Compatible almost CR structures}
\label{subsec:compatible-almost-CR}

Let $(M,H)$ be a contact manifold of dimension $2n-1$, where $n\ge 2$.
An \emph{almost CR structure} $\gamma$ on the contact distribution $H$ means
a smooth section of $\End(H)$ satisfying $\gamma^2=-\id_H$.
We say that $\gamma$ is \emph{compatible} when the \emph{Levi form}
with respect to a contact 1-form $\theta$,
\begin{equation}
	h_{\theta,\gamma}(V,W):=d\theta(V,\gamma W),\qquad
	V,\,W\in H,
\end{equation}
is symmetric and has definite signature.
The condition is irrelevant to the choice of the 1-form $\theta$ because $h_{f\theta,\gamma}=fh_{\theta,\gamma}$.

The Levi form is symmetric if and only if
\begin{equation}
	\label{eq:partial-integrability}
	[\Gamma(T^{1,0}_\gamma M),\Gamma(T^{1,0}_\gamma M)]\subset\Gamma(T^{1,0}_\gamma M\oplus\conj{T^{1,0}_\gamma M}),
\end{equation}
where $H_\mathbb{C}=T^{1,0}_\gamma M\oplus\conj{T^{1,0}_\gamma M}$ is the eigendecomposition of
the complexification of $H$ with respect to $\gamma$,
as is easily seen from $d\theta(V,\gamma W)=-\theta([V,\gamma W])$.
In particular, the Levi form is always symmetric for integrable almost CR structures.
Condition \eqref{eq:partial-integrability} is called the \emph{partial integrability} in the literature
(e.g., \cites{Cap-Schichl-00,Cap-Slovak-09,Matsumoto-14,Matsumoto-16,Matsumoto-preprint16}),
but it must be noted that the partial integrability is determined pointwisely.

On the other hand, the definiteness of $h_{\theta,\gamma}$ is usually referred to
 as the \emph{strict pseudoconvexity} of $\gamma$.
Therefore our ``compatible almost CR structures'' are the same as
``strictly pseudoconvex partially integrable almost CR structures.''
Our new terminology is meant for brevity and to avoid possible confusion.
Compatible almost CR structures are generically non-integrable if $n\ge 3$.

In what follows,
when $\gamma$ is a compatible almost CR structure, we always (implicitly) choose $\theta$
so that $h_{\theta,\gamma}$ is positive definite.
For each fixed $\gamma$,
there is a one-to-one correspondence between such contact forms and
representative metrics of the conformal class $[h_{\theta,\gamma}]$ of metrics of $H$.

A contact form $\theta$ determines the \emph{Reeb vector field} $T$, which is transverse to $H$,
by the following conditions: $d\theta(T,\mathord{\cdot})=0$ and $\theta(T)=1$.

\subsection{ACH metrics}
\label{subsec:ACH-metrics}

Let $\overline{X}$ be a compact smooth manifold-with-boundary of dimension $2n$, where $n\ge 2$,
and let $X$ be its interior. The boundary is denoted by $\bdry X$.
We assume that $\bdry X$ is equipped with a contact distribution $H$,
and the set of smooth (resp.\ $C^{k,\alpha}$) compatible almost CR structures of $H$ is
denoted by $\mathcal{C}_H$ (resp.\ $\mathcal{C}^{k,\alpha}_H$).

The most general definition of ACH metrics can be stated as follows.
Note that, for technical reasons, when we simply refer to an ACH metric $g$ in this paper,
we allow $g$ not to be a smooth Riemannian metric on $X$
(see also Definition \ref{dfn:ACH-metric-of-Holder-class}).

\begin{dfn}
	\label{dfn:ACH-metrics}
	For $\gamma\in\mathcal{C}^{k,\alpha}_H$ and a contact form $\theta$,
	we define the metric $g_{\theta,\gamma}$ on $\bdry X\times(0,\varepsilon)_x$ by
	\begin{equation}
		\label{eq:model-metric}
		g_{\theta,\gamma}=\frac{1}{2}\left(4\frac{dx^2}{x^2}+\frac{\theta^2}{x^4}+\frac{h_{\theta,\gamma}}{x^2}\right),
	\end{equation}
	where we extend $h_{\theta,\gamma}$ to $T\bdry X$ by setting $h_{\theta,\gamma}(T,\mathord{\cdot})=0$
	for the Reeb vector field $T$.
	A Riemannian metric $g$ on $X$ is called an \emph{ACH metric} with \emph{conformal infinity} $\gamma$
	when $g$ is asymptotic to $g_{\theta,\gamma}$ for some contact form $\theta$,
	in the sense that there exists a diffeomorphism $\Phi$
	from an open neighborhood $\mathcal{U}$ of $\bdry X\subset\overline{X}$ to $\bdry X\times[0,\varepsilon)_x$
	such that the following are satisfied:
	\begin{enumerate}[(i)]
		\item
			$\Phi|_{\bdry X}$ is the identity map on $\bdry X$, where
			$\bdry X\times\set{0}\subset\bdry X\times[0,\varepsilon)$ is identified with $\bdry X$.
		\item
			$\abs{g-\Phi^*g_{\theta,\gamma}}_{\Phi^*g_{\theta,\gamma}}$ uniformly tends to zero as $x\to 0$.
	\end{enumerate}
	We call such $\Phi$ an \emph{admissible collar neighborhood diffeomorphism} of an ACH metric $g$
	with respect to $\theta$.
\end{dfn}

\begin{rem}
	\label{rem:Theta-tangent-bundles}
	When $g$ is an ACH metric,
	an admissible collar neighborhood diffeomorphism exists for any choice of $\theta$.
	This is because, for $\theta$ and $\Hat{\theta}=e^{2u}\theta$,
	the model metrics $g_{\theta,\gamma}$ and $g_{\Hat{\theta},\gamma}$
	are asymptotic to each other in the sense that, if we identify
	neighborhoods of the boundaries of two copies of $\bdry X\times[0,\varepsilon)$
	by a certain diffeomorphism that restricts to $\id_{\bdry X}$,
	then the difference between $g_{\theta,\gamma}$ and $g_{\Hat{\theta},\gamma}$ tends to $0$ uniformly.
	Namely, a diffeomorphism
	\begin{equation*}
		\Psi=\Psi(q,x)=(\psi(q,x),\Hat{x}(q,x))\colon\mathcal{U}\to\Hat{\mathcal{U}},
	\end{equation*}
	where
	$\mathcal{U}$ and $\Hat{\mathcal{U}}$ are open neighborhoods of the boundary of $\bdry X\times[0,\varepsilon)$,
	has the desired property if and only if
	\begin{equation}
		\label{eq:collar-neighborhood-equivalence-1}
		\psi(\mathord{\cdot},0)=\id_{\bdry X},\qquad
		\left.\frac{\partial\Hat{x}}{\partial x}\right|_{\bdry X\times\set{0}}=e^u,
	\end{equation}
	and
	\begin{equation}
		\label{eq:collar-neighborhood-equivalence-2}
		(d\psi)_{(q,x)}(\partial_x)=xY+O(x^2) \qquad\text{for some $Y\in H_q$}.
	\end{equation}
	Such a diffeomorphism $\Psi$ certainly exists---for example, we can take $\Psi(q,x)=(q,xe^{u(q)})$.

	Let $\set{Y_1,\dots,Y_{2n-2}}$ be any local frame of $H$.
	The conditions \eqref{eq:collar-neighborhood-equivalence-1} and \eqref{eq:collar-neighborhood-equivalence-2} imply
	that the functions describing the change of basis from (the push-forward of)
	$\set{x\partial_x,x^2T,xY_1,\dots,xY_{2n-2}}$
	to $\set{\Hat{x}\partial_{\Hat{x}},\Hat{x}^2\Hat{T},\Hat{x}Y_1,\dots,\Hat{x}Y_{2n-2}}$,
	where $T$ and $\Hat{T}$ are the Reeb vector fields of $\theta$ and $\Hat{\theta}$, respectively,
	are smooth up to the boundary, and the boundary values are given by
	\begin{equation*}
		\begin{pmatrix}
			1 \\
			& 1 \\
			& & e^u \\
			& & & \ddots \\
			& & & & e^u
		\end{pmatrix}.
	\end{equation*}
	This implies that the set $\set{x\partial_x,x^2T,xY_1,\dots,xY_{2n-2}}$ of vector fields
	spans a vector bundle over $\overline{X}$ that does not depend on $\theta$ nor $\Phi$.
	It is the underlying \emph{$\Theta$-tangent bundle} of an ACH manifold,
	which is due to Epstein, Melrose, and Mendoza \cite{Epstein-Melrose-Mendoza-91}.
\end{rem}

We also need subtler definitions of some classes of ACH metrics.
We first introduce the notion of ACH metrics that are smooth up to the boundary.
In the definition below, $\set{Y_1,\dots,Y_{2n-2}}$ is any local frame of $H$.

\begin{dfn}
	\label{dfn:smooth-ACH-metric}
	An ACH metric $g$ on $X$ with conformal infinity $\gamma\in\mathcal{C}_H$
	is said to be \emph{smooth up to the boundary}
	when, for some $\theta$, if $T$ is the Reeb vector field of $\theta$ and
	$\Phi\colon\mathcal{U}\to\bdry X\times[0,\varepsilon)_x$ is an admissible collar neighborhood diffeomorphism
	with respect to $\theta$,
	then the components of $(\Phi^{-1})^*g$ with respect to
	$\set{x\partial_x,x^2T,xY_1,\dots,xY_{2n-2}}$ are smooth up to the boundary.
\end{dfn}

Remark \ref{rem:Theta-tangent-bundles} implies that,
if $g$ is smooth up to the boundary as an ACH metric,
then the required smoothness of the components holds for any $\theta$ and any $\Phi$.
Simply put, such an ACH metric is a smooth metric of the $\Theta$-tangent bundle.

Next, let $E=\Sym^2T^*X$ be the bundle of symmetric 2-tensors over $X$.
Given an ACH metric $g$ that is smooth up to the boundary,
$C^{k,\alpha}(X,E)$ denotes the H\"older space of $C^k$ sections of $E$ with bounded $C^{k,\alpha}$ norm
with respect to $g$.
We may alternatively use the ACH version of M\"obius charts \cite{Roth-99-Thesis}*{Section 2.5}
to define the H\"older norm (this approach is taken by Lee \cite{Lee-06}*{Chapters 2 and 3} in the AH setting),
which in particular implies that the norms are equivalent for any choice of $g$
that accepts the same admissible collar neighborhood diffeomorphism $\Phi$.

For $\delta\in\mathbb{R}$, the weighted H\"older space is defined by
\begin{equation*}
	C^{k,\alpha}_\delta(X,E)=x^\delta C^{k,\alpha}(X,E).
\end{equation*}

The notation above is used for other $\GL(2n)$-invariant subbundles
of $(TX)^{\otimes r}\otimes(T^*X)^{\otimes s}$ as well,
or under the existence of some fixed ACH metric, for its $O(2n)$-invariant subbundles.
Furthermore, if there is some fixed ACH almost Hermitian structure (introduced
in Section \ref{subsec:ACH-almost-Hermitian-structures}), the notation can also be used
for $U(n)$-invariant subbundles of $(T_{\mathbb{C}}X)^{\otimes r}\otimes(T_{\mathbb{C}}^*X)^{\otimes s}$.

\begin{dfn}
	\label{dfn:ACH-metric-of-Holder-class}
	For $\delta\in(0,1]$, an ACH metric $g$ on $X$ with conformal infinity $\gamma\in\mathcal{C}^{k,\alpha}_H$
	is said to be \emph{of class $C^{k,\alpha}_\delta$} if $g$ is locally $C^{k,\alpha}$ in $X$
	and it can be expressed as
	\begin{equation*}
		g=g_{\theta,\gamma}+\sigma,\qquad \sigma\in C^{k,\alpha}_\delta(X,\Sym^2 T^*X),
	\end{equation*}
	where $g_{\theta,\gamma}$ is the model metric \eqref{eq:model-metric} pulled back
	by an admissible collar neighborhood diffeomorphism $\Phi\colon\mathcal{U}\to\bdry X\times[0,\varepsilon)_x$
	and extended arbitrarily to the whole $X$.
	The set of all ACH metrics on $X$ of class $C^{k,\alpha}_\delta$ is denoted by $\mathcal{M}^{k,\alpha}_\delta$.
\end{dfn}

\begin{rem}
	By ``extended arbitrarily to the whole $X$'' in the above definition,
	we mean that $g_{\theta,\gamma}$, originally defined in $\mathcal{U}\setminus\bdry X$,
	is extended to a $C^{k,\alpha}$ Riemannian metric
	\begin{equation*}
		\chi_1 g_{\theta,\gamma}+\chi_2 h
	\end{equation*}
	on $X$, where $h$ is another $C^{k,\alpha}$ Riemannian metric on $X$ and
	$\set{\chi_1, \chi_2}$ is a smooth partition of unity on $\overline{X}$
	subordinate to $\set{\mathcal{U}, X\setminus\mathcal{U}'}$,
	$\mathcal{U}'$ being an open neighborhood of $\bdry X$
	that is contained and relatively compact in $\mathcal{U}$.
	The same expression will be used several times in the sequel.
	When we say that the extension is simultaneously done for all $\gamma$
	in a subset of $\mathcal{C}^{k,\alpha}_H$, we mean that the extension as above is made
	with $\gamma$-independent $h$ and $\set{\chi_1, \chi_2}$.
\end{rem}

On any bounded strictly pseudoconvex domain $\Omega$ in a Stein manifold,
the Cheng--Yau metric \cite{Cheng-Yau-80} is an ACH metric of class $C^{k,\alpha}_1$ for any $k$ and $\alpha$,
if $\overline{X}$ is taken to be the square root of $\overline{\Omega}$ in the sense of
Epstein--Melrose--Mendoza \cite{Epstein-Melrose-Mendoza-91}.
This is because, if we express the metric in terms of a K\"ahler potential $\log(1/\varphi)$,
then $\varphi$ has polyhomogeneous expansion at the boundary that involves only logarithmic singularity,
as shown by Lee and Melrose \cite{Lee-Melrose-82}.
(For the same reason, the Bergman metric on any bounded strictly pseudoconvex domain in $\mathbb{C}^n$
is also an ACH metric of class $C^{k,\alpha}_1$ for any $k$ and $\alpha$,
owing to the result of Fefferman \cite{Fefferman-74}.)

Finally, we introduce the notion of smooth families of elements of $\mathcal{M}^{k,\alpha}_\delta$.

\begin{dfn}
	A family of ACH metrics $g_\gamma\in\mathcal{M}^{k,\alpha}_\delta$
	parametrized by the conformal infinity $\gamma\in\mathcal{U}$ of each $g_\gamma$,
	where $\mathcal{U}$ is an open set of $\mathcal{C}^{k,\alpha}_H$, is \emph{smooth}
	if $g_\gamma$ can be expressed as
	\begin{equation*}
		g_\gamma=g_{\theta,\gamma}+\sigma_\gamma,
	\end{equation*}
	where $g_{\theta,\gamma}$ and $\sigma_\gamma$ have the following properties:
	\begin{enumerate}[(i)]
		\item
			\label{item:smoothness-of-family-of-smooth-ACH-metrics}
			$g_{\theta,\gamma}$ is the model metric \eqref{eq:model-metric} pulled back by
			a $\gamma$-independent admissible collar neighborhood $\Phi$ and extended to the whole $X$
			simultaneously for all $\gamma\in\mathcal{U}$.
		\item
			$\sigma_\gamma$ belongs to $C^{k,\alpha}_\delta$ for all $\gamma$,
			and the mapping $\gamma\mapsto\sigma_\gamma$ is smooth.
	\end{enumerate}
\end{dfn}

\subsection{ACH almost Hermitian structures}
\label{subsec:ACH-almost-Hermitian-structures}

As above, let $\overline{X}$ be a compact smooth manifold-with-boundary of dimension $2n$
whose boundary $\bdry X$ is equipped with a contact distribution $H$.
Let $g$ be an ACH metric on $X$.
We set up our terminology regarding extensions of the conformal infinity $\gamma$
into almost complex structures compatible with $g$,
emphasizing the parallelism with the previous subsection.

Again, we start with a standard model on $\bdry X\times(0,\varepsilon)_x$.
For $\gamma\in\mathcal{C}^{k,\alpha}_H$ and a contact form $\theta$,
the almost complex structure $J_{\theta,\gamma}$ is defined as follows.
Let $\set{Z_1,\dots,Z_{n-1}}$ be a local frame of the CR holomorphic tangent bundle
$T^{1,0}\bdry X\subset H_\mathbb{C}$, which is the $i$-eigenbundle of $\gamma$,
and let $\set{Z_{\conj{1}},\dots,Z_{\conj{n-1}}}$ be its complex conjugate.
Let $T$ denote the Reeb vector field of $\theta$. We set
\begin{alignat*}{3}
	J_{\theta,\gamma}\left(\frac{1}{2}x\partial_x\right)&=x^2T,&\qquad
	J_{\theta,\gamma}(x^2T)&=-\frac{1}{2}x\partial_x,\\
	J_{\theta,\gamma}(xZ_\alpha)&=ixZ_\alpha,&\qquad
	J_{\theta,\gamma}(xZ_{\conj{\alpha}})&=-ixZ_{\conj{\alpha}},&\qquad
	\alpha&=1,\,\dots,\,n-1.
\end{alignat*}
This is in fact compatible with the model metric $g_{\theta,\gamma}$ given by \eqref{eq:model-metric}.
Alternatively, we can say that $J_{\theta,\gamma}$
is the almost complex structure whose holomorphic tangent bundle is spanned by
$\set{\bm{Z}_0,\bm{Z}_1,\dots,\bm{Z}_{n-1}}$, where
\begin{equation*}
	\bm{Z}_0=\frac{1}{2}x\partial_x+ix^2T\qquad\text{and}\qquad
	\bm{Z}_\alpha=xZ_\alpha,\quad
	\alpha=1,\,\dots,\,n-1.
\end{equation*}

\begin{dfn}
	\label{dfn:ACH-almost-Hermitian}
	An \emph{ACH almost Hermitian structure} with \emph{conformal infinity} $\gamma\in\mathcal{C}^{k,\alpha}_H$
	is a pair $(g,J)$ comprising a Riemannian metric $g$
	and a compatible almost complex structure $J$ that is asymptotic to $(g_{\theta,\gamma},J_{\theta,\gamma})$
	for some contact form $\theta$,
	in the sense that there exists a diffeomorphism $\Phi$
	from an open neighborhood $\mathcal{U}$ of $\bdry X\subset\overline{X}$ to $\bdry X\times[0,\varepsilon)_x$
	such that the following are satisfied:
	\begin{enumerate}[(i)]
		\item
			$\Phi|_{\bdry X}$ is the identity map on $\bdry X$, where
			$\bdry X\times\set{0}\subset\bdry X\times[0,\varepsilon)$ is identified with $\bdry X$.
		\item
			Both $\abs{g-\Phi^*g_{\theta,\gamma}}_{\Phi^*g_{\theta,\gamma}}$
			and $\abs{J-\Phi^*J_{\theta,\gamma}}_{\Phi^*g_{\theta,\gamma}}$ uniformly tend to zero as $x\to 0$.
	\end{enumerate}
	We call such $\Phi$ an \emph{admissible collar neighborhood diffeomorphism}
	of an ACH almost Hermitian structure $(g,J)$
	with respect to $\theta$.
\end{dfn}

\begin{dfn}
	\label{dfn:smooth-ACH-almost-Hermitian}
	An ACH almost Hermitian structure $(g,J)$ with conformal infinity $\gamma\in\mathcal{C}_H$
	is said to be \emph{smooth up to the boundary} when,
	for some $\theta$, if $T$ is the Reeb vector field of $\theta$ and
	$\Phi\colon\mathcal{U}\to\bdry X\times[0,\varepsilon)_x$ is
	some admissible collar neighborhood diffeomorphism with respect to $\theta$,
	then the components of $(\Phi^{-1})^*g$ and those of $(\Phi^{-1})^*J$
	with respect to $\set{x\partial_x,x^2T,xY_1,\dots,xY_{2n-2}}$ are smooth up to the boundary.
\end{dfn}

Remark \ref{rem:Theta-tangent-bundles} implies that
Definitions \ref{dfn:ACH-almost-Hermitian} and \ref{dfn:smooth-ACH-almost-Hermitian}
remain equivalent if we replace ``for some contact form $\theta$'' with ``for any contact form $\theta$.''

The notions of ACH almost Hermitian structures of class $C^{k,\alpha}_\delta$
and smooth families thereof are introduced just like those of ACH metrics.

\begin{dfn}
	\label{dfn:ACH-almost-Hermitian-structure-of-Holder-class}
	For $\delta\in(0,1]$,
	an ACH almost Hermitian structure $(g,J)$ on $X$
	is said to be \emph{of class $C^{k,\alpha}_\delta$} if $g$ and $J$ are locally $C^{k,\alpha}$ in $X$
	and they can be expressed as
	\begin{equation*}
		\begin{cases}
			g=g_{\theta,\gamma}+\sigma,\qquad \sigma\in C^{k,\alpha}_\delta(X,\Sym^2T^*X),\\
			J=J_{\theta,\gamma}+\psi,\qquad \psi\in C^{k,\alpha}_\delta(X,\End(TX)),
		\end{cases}
	\end{equation*}
	where $g_{\theta,\gamma}$ and $J_{\theta,\gamma}$ are the model metric
	and the model almost complex structure given above, pulled back by an admissible collar neighborhood
	diffeomorphism $\Phi$ and extended arbitrarily to the whole $X$.
	The set of all ACH almost Hermitian structures on $X$ of class $C^{k,\alpha}_\delta$ is
	denoted by $\tilde{\mathcal{M}}^{k,\alpha}_\delta$.
\end{dfn}

We note that the extended $(g_{\theta,\gamma},J_{\theta,\gamma})$ in the above definition
does not need to be an almost Hermitian structure on $X$;
$g_{\theta,\gamma}$ and $J_{\theta,\gamma}$ can be independently extended,
and even the extension of $J_{\theta,\gamma}$ allow not to be an almost complex structure.

\begin{dfn}
	A family of ACH almost Hermitian structures $(g_\gamma,J_\gamma)\in\tilde{\mathcal{M}}^{k,\alpha}_\delta$
	parametrized by the conformal infinity $\gamma\in\mathcal{U}$,
	where $\mathcal{U}$ is an open set of $\mathcal{C}^{k,\alpha}_H$, is \emph{smooth}
	if $g_\gamma$ and $J_\gamma$ can be expressed as
	\begin{equation*}
		g_\gamma=g_{\theta,\gamma}+\sigma_\gamma\qquad\text{and}\qquad
		J_\gamma=J_{\theta,\gamma}+\psi_\gamma
	\end{equation*}
	where:
	\begin{enumerate}[(i)]
		\item
			\label{item:smoothness-of-family-of-smooth-ACH-almost-Hermitian-structures}
			$g_{\theta,\gamma}$ and $J_{\theta,\gamma}$ are the model metric
			and the model almost complex structure
			pulled back by a $\gamma$-independent admissible collar neighborhood $\Phi$
			and extended to the whole $X$ simultaneously for all $\gamma\in\mathcal{U}$.
		\item
			$\sigma_\gamma\in C^{k,\alpha}_\delta$ and $\psi_\gamma\in C^{k,\alpha}_\delta$ for all $\gamma$,
			and the mappings $\gamma\mapsto\sigma_\gamma$ and $\gamma\mapsto\psi_\gamma$ are both smooth.
	\end{enumerate}
\end{dfn}

\subsection{The Fredholm theorem}

Our proof of Theorem \ref{thm:global-existence} will be made possible by the following Fredholm theorem.

Suppose that $E$ is a vector bundle over an almost Hermitian manifold $(X,g,J)$
of the form
$(T_\mathbb{C}X)^{\otimes r}\otimes (T^*_\mathbb{C}X)^{\otimes s}$,
or its $U(n)$-invariant subbundle, or the direct sum of such bundles.
A linear differential operator $P\colon\Gamma(E)\to\Gamma(E)$ is called \emph{geometric} of \emph{order} $m$
if $Pu$ is given by a universal expression ``of order $m$'' in terms of the Ehresmann--Libermann connection $\nabla$,
that is, as the sum of contractions of tensor products of
$\nabla^lu$, $\nabla^{l'-2}R$, $\nabla^{l'-1}N$, $\nabla^{l'-1}T$,
and $g$, $g^{-1}$, where $l\le m$ and $l'$ (which can be different from factor to factor) satisfies $l'\le m-l$.
If $k\ge m$ and $(g,J)$ is an ACH almost Hermitian structure of class $C^{k,\alpha}_\nu$ for some $\nu\in(0,1]$,
then any geometric linear differential operator $P\colon\Gamma(E)\to\Gamma(E)$ of
order $m$ naturally defines a bounded operator
\begin{equation*}
	C^{k,\alpha}_\delta(X,E)\to C^{k-m,\alpha}_\delta(X,E)
\end{equation*}
for an arbitrary $\delta\in\mathbb{R}$.
(This definition can be readily generalized to differential operators between different vector bundles,
but we do not need to do so in this paper.)

Note that every vector bundle $E$ of the type described above is
expressed as $E=\mathcal{F}\times_{U(n)}V$, where $\mathcal{F}$ is the unitary frame bundle over $X$ and
$V$ is some $U(n)$-representation.
Then the representation $V$ simultaneously defines the vector bundle, which we express by the same symbol $E$,
over all almost Hermitian manifolds.
In particular, we have the bundle $E$ defined over $\mathbb{C}H^n$, and this is needed to state the next theorem.

Another ingredient that we need is the
notion of \emph{indicial roots} associated with any geometric linear differential operator $P$,
which does not depend on $(X,g,J)$ but only on the universal expression of $P$.
We refer the reader to \cite{Matsumoto-preprint16}*{Section 1} for the definition.
Although \cite{Matsumoto-preprint16} considered geometric operators associated with ACH metrics
rather than ACH almost Hermitian structures,
the definition of the indicial roots there applies to our case without any change.
Let $\Sigma_P\subset\mathbb{C}$ be the set of indicial roots; then,
\begin{equation*}
	R_P:=\min_{s\in\Sigma_P}\abs{\Re s-n}\ge 0
\end{equation*}
is called the \emph{indicial radius} of $P$.

\begin{thm}
	\label{thm:Fredholm}
	Let $X$ be equipped with an ACH almost Hermitian structure of class $C^{k,\alpha}_\nu$ for some $\nu\in(0,1]$.
	Let $P\colon\Gamma(X,E)\to\Gamma(X,E)$ be
	a formally self-adjoint geometric elliptic linear differential operator of order $m$,
	and assume that it satisfies, on $\mathbb{C}H^n$,
	\begin{equation}
		\label{eq:coerciveness}
		\norm{u}_{L^2}\le C\norm{Pu}_{L^2},\qquad
		u\in\dom P\subset L^2(\mathbb{C}H^n,E)
	\end{equation}
	for some constant $C>0$,
	where $\dom P$ denotes the domain of the maximal closed extension of $P$
	as an unbounded operator $L^2\to L^2$.
	Then, for $k\ge m$, the bounded operator
	\begin{equation*}
		P\colon C^{k,\alpha}_\delta(X,E)\to C^{k-m,\alpha}_\delta(X,E)
	\end{equation*}
	is a Fredholm operator of index zero if $n-R_P<\delta<n+R_P$, where $R_P$ is the indicial radius of $P$.
	Moreover, the kernel of $P$ within this range of $\delta$ equals the $L^2$ kernel $\ker_{(2)}P$.
\end{thm}

Versions of this theorem are given by Roth \cite{Roth-99-Thesis}*{Proposition 4.15}
and by Biquard \cite{Biquard-00}*{Proposition I.3.5}.
The expositions in these two references are restricted to certain second-order operators,
but they are straightforwardly extended to general geometric operators \emph{associated with ACH metrics}.
Theorem \ref{thm:Fredholm} is more general than that, for it concerns geometric operators
\emph{associated with ACH almost Hermitian structures}.
The modification to the proof is minor, but it may not be totally trivial.
We illustrate it in the following sketch of the proof.

Note that we omit the proof of the assertions on the Fredholm index of $P$ and that the kernel equals $\ker_{(2)}P$.
We can follow \cite{Biquard-00}*{pp.\ 34--35} or \cite{Lee-06}*{pp.\ 50--56} for this part.

\begin{proof}[Sketch of the proof of the Fredholm property in Theorem \ref{thm:Fredholm}]
	The assumption \eqref{eq:coerciveness}
	implies that the operator $P$ defines an isomorphism $H^m(\mathbb{C}H^n,E)\to L^2(\mathbb{C}H^n,E)$,
	where $H^m(\mathbb{C}H^n,E)$ is the $L^2$-Sobolev space of exponent $m$,
	as (more or less) discussed in
	\cite{Roth-99-Thesis}*{Proposition 4.8}, \cite{Biquard-00}*{Sections I.2.B and I.2.C}.
	Let $P^{-1}\colon L^2\to H^m$ be the inverse.
	Then, the boundary asymptotic behavior of the Green kernel of $P$ (i.e., the Schwartz kernel of $P^{-1}$)
	on $\mathbb{C}H^n$ can be determined using the indicial polynomial of $P$,
	and we conclude that $P$ also defines an isomorphism
	$C^{k,\alpha}_\delta\to C^{k-m,\alpha}_\delta$ in the range $n-R_P<\delta<n+R_P$.
	The proof of this fact is basically given in
	\cite{Roth-99-Thesis}*{Proposition 5.9}, \cite{Biquard-00}*{Proposition I.2.5};
	see also \cite{Lee-06}*{Chapter 5}.

	Next we want to introduce the ACH version of boundary M\"obius charts.
	The idea is essentially given in \cite{Roth-99-Thesis}*{Section 2.4},
	and also in \cite{Biquard-00}, in the beginning of the proof of the Fredholm property in Proposition I.3.5.
	For the AH case, see Lemma 6.1 in \cite{Lee-06}.

	For this, we identify $\mathbb{C}H^n$ with the Siegel upper-half space
	$\set{(z,w)\in\mathbb{C}^{n-1}\times\mathbb{C}|\Im w>\abs{z}^2}$,
	which we write $X_0$.
	We equip $\overline{X}_0=\set{\Im w\ge\abs{z}^2}$ with the square-root smooth structure:
	we set $r=\Im w-\abs{z}^2$, $t=\Re w$ and $x=\sqrt{r/2}$,
	by which $(x,z,t)$ is a smooth global coordinate system on $\overline{X}_0$.
	The complex hyperbolic metric and the standard complex structure on $X_0$ are
	denoted by $g_0$ and $J_0$, respectively.

	We take coordinate neighborhoods in $\overline{X}$ near $\bdry X$ modelled
	on open neighborhoods of $(0,0)\in\overline{X}_0$ such that $(g,J)$ is close to $(g_0,J_0)$ on them.
	Compared to the arguments in \cite{Roth-99-Thesis} and \cite{Biquard-00},
	we make two minor adjustments here:
	an estimate for $J$ is given, which is straightforward,
	and $\varepsilon^\nu$ is used instead of $\varepsilon$,
	which is due to our definition of ACH metrics/almost Hermitian structures.
	Let a smooth boundary defining function $\tilde{x}$ of $\overline{X}$ be fixed,
	and for $\varepsilon>0$,
	we set $\mathcal{U}_\varepsilon=\set{0\le\tilde{x}<\varepsilon}$
	and $\overline{\mathcal{U}}_\varepsilon=\set{0\le\tilde{x}\le\varepsilon}$.
	Then, for each point $q\in\bdry X$, we can take a diffeomorphism $\Phi_q\colon U_q\to V_1$,
	where $U_q$ is an open neighborhood of $q$ in $\overline{X}$ and
	\begin{equation*}
		V_\rho=\set{(x,z,t)|x<\rho,\,\abs{z}<\rho,\,\abs{t}<\rho^2}
		\subset\overline{X}_0,
	\end{equation*}
	such that $\mathcal{A}_\varepsilon=\set{U_q}_{q\in\bdry X}$ covers $\mathcal{U}_\varepsilon$ and
	\begin{gather*}
		\norm{(\Phi_q)_*g-g_0}_{C^{k,\alpha}(V_1)}<C\varepsilon^\nu,\qquad
		\sup_{V_1}\abs{((\Phi_q)_*g)^{-1}}<C,\\
		\norm{(\Phi_q)_*J-J_0}_{C^{k,\alpha}(V_1)}<C\varepsilon^\nu
	\end{gather*}
	are satisfied, where the constant $C>0$ is independent of $q\in\bdry X$
	and of $\varepsilon>0$. The norms on the left-hand sides are defined via $g_0$.

	For any section $u\in C^{k,\alpha}_\delta(X,E)$ supported in $U_q$,
	we can establish an estimate of
	$Pu-\mathring{P}_qu$,
	where we define $\mathring{P}_q$ by ``implanting'' $P$ of $\mathbb{C}H^n$ onto $U_q$, as follows.
	The argument here corresponds to \cite{Roth-99-Thesis}*{Proposition 4.14} and
	\cite{Biquard-00}*{Equation (3.7)}, and \cite{Lee-06}*{Equation (6.5)} in the AH case.

	In order to define $\mathring{P}_qu$,
	we need to introduce an identification of the bundles $E|_{U_q}$ and $E|_{V_1}$.
	Let $\set{\bm{Z}_0,\bm{Z}_1,\dots,\bm{Z}_{n-1}}$ be some fixed unitary frame of $T^{1,0}\mathbb{C}H^n$ over $V_1$.
	Then each $(\Phi_q)_*\bm{Z}_i$ is a section of $T_\mathbb{C}X=T^{1,0}X\oplus\conj{T^{1,0}X}$ over $U_q$,
	and we write its projection to the first summand as $\bm{W}_i^0$.
	We then apply the Gram--Schmidt process to $\set{\bm{W}_0^0,\bm{W}_1^0,\dots,\bm{W}_{n-1}^0}$
	to obtain a unitary frame
	$\set{\bm{W}_0,\bm{W}_1,\dots,\bm{W}_{n-1}}$.
	If we write $E=\mathcal{F}\times_{U(n)}V$,
	then by the frame $\set{\bm{W}_0,\bm{W}_1,\dots,\bm{W}_{n-1}}$,
	any section of $E|_{U_q}$ can be seen as a function with values in $V$.
	We define $\Psi_qu$ as
	the section of $E|_{V_1}$ given by the same function with respect to the frame
	$\set{\bm{Z}_0,\bm{Z}_1,\dots,\bm{Z}_{n-1}}$.

	Using this identification, the operator $\mathring{P}_q$ acting on sections of $E|_{U_q}$ is defined by
	\begin{equation*}
		\mathring{P}_qu=\Psi_q^{-1}(P(\Psi_qu)).
	\end{equation*}
	Then it can be shown that
	\begin{equation*}
		\norm{Pu-\mathring{P}_qu}_{C^{k-m,\alpha}_\delta(U_q)}
		\le C\varepsilon^\nu\norm{u}_{C^{k,\alpha}_\delta(U_q)},
	\end{equation*}
	where $C>0$ does not depend on $q\in\bdry X$ or $\varepsilon>0$.
	
	At this point, we need to look back and re-examine our construction of the family
	$\mathcal{A}_\varepsilon=\set{U_q}_{q\in\bdry X}$.
	The construction of $\mathcal{A}_\varepsilon$ can be carried out so that
	there exists $N\in\mathbb{N}$, which is independent of $\varepsilon$, with the following property:
	\begin{quote}
		For any $\varepsilon>0$, there is a finite subfamily
		$\mathcal{A}'_\varepsilon=\set{U_{q_\lambda}}_{\lambda\in\Lambda}$ of $\mathcal{A}_\varepsilon$ such that
		$\mathcal{U}_{\varepsilon/2}$ is covered by $\set{\Phi_{q_\lambda}^{-1}(V_{1/2})}$, and,
		for any $p\in\mathcal{U}_{\varepsilon}$,
		the number of $\lambda\in\Lambda$ for which $p\in U_{q_\lambda}$ is at most $N$.
	\end{quote}
	In the AH case, this is discussed in \cite{Lee-06}*{p.~47} by referring to its Lemma 2.2,
	and the same argument applies to our case as well.
	We take a subfamily $\mathcal{A}'_\varepsilon$ having the property above for a particular $\varepsilon$
	that is specified later,
	and, instead of writing $\Phi_{q_\lambda}$, $U_{q_\lambda}$, and $\mathring{P}_{q_\lambda}$,
	we simply write $\Phi_\lambda$, $U_\lambda$, and $\mathring{P}_\lambda$, respectively.

	Let $\varphi\colon V_1\to[0,1]$ be a smooth bump function that equals $1$ in $V_{1/2}$ and
	$0$ outside $V_{3/4}$.
	Let $\varphi_\lambda=\Phi_\lambda^*\varphi$, which is supported in $U_q$.
	Moreover, let $\varphi_0\colon X\to[0,1]$ be a smooth bump function
	supported in $X\setminus\overline{\mathcal{U}}_{\varepsilon/4}$
	that equals $1$ in $X\setminus\mathcal{U}_{\varepsilon/2}$. We set
	\begin{equation*}
		\chi_\lambda=\frac{\varphi_\lambda}{\displaystyle\sqrt{\varphi_0^2+\sum_{\lambda\in\Lambda}\varphi_\lambda^2}}
		\qquad
		\text{and}\qquad
		\chi_0=\frac{\varphi_0}{\displaystyle\sqrt{\varphi_0^2+\sum_{\lambda\in\Lambda}\varphi_\lambda^2}}.
	\end{equation*}
	Then, $\set{\chi_\lambda^2}_{\lambda\in\Lambda}\cup\set{\chi_0^2}$ is a smooth partition of unity
	subordinate to the covering
	$\set{U_\lambda}_{\lambda\in\Lambda}\cup\set{X\setminus\overline{\mathcal{U}}_{\varepsilon/4}}$ of $X$.
	We also have
	\begin{equation*}
		\norm{\chi_\lambda}_{C^{k,\alpha}}<C\quad\text{and}\quad
		\norm{\chi_0}_{C^{k,\alpha}}<C
	\end{equation*}
	for some $C>0$ that is independent of $\varepsilon$,
	and $\nabla\chi_\lambda\in C^{k-1,\alpha}_1$ (and $\nabla\chi_0\in C^{k-1,\alpha}_1$, which is obvious).
	See also similar arguments that can be found in the proof of \cite{Roth-99-Thesis}*{Proposition 4.15},
	the proof of the Fredholm property of \cite{Biquard-00}*{Proposition I.3.5},
	and \cite{Lee-06}*{p.\ 47} in the AH case.

	In order to prove that $P$ is a Fredholm operator,
	it suffices to show that
	$Q_1Pu=u+K_1u$ and $PQ_2u=u+K_2u$
	for some bounded operators $Q_1$, $Q_2$ and compact operators $K_1$, $K_2$.
	We construct such $Q_1$ and $Q_2$ using the approach of \cite{Biquard-00}.

	Take a parametrix $Q_0$ for $P$ restricted to sections supported in
	$X\setminus\overline{\mathcal{U}}_{\varepsilon/4}$.
	We define the bounded operator $Q\colon L^2\to H^m$,
	which is also a bounded operator as $C_\delta^{k-m,\alpha}\to C_\delta^{k,\alpha}$ for $n-R_P<\delta<n+R_P$, by
	\begin{equation*}
		Qu=\sum_{\lambda\in\Lambda}\chi_\lambda\mathring{P}_\lambda^{-1}(\chi_\lambda u)+\chi_0Q_0(\chi_0u).
	\end{equation*}
	Then, if we write $Q_0P=I+K_0$, we obtain
	\begin{equation*}
		QPu=u+\sum_{\lambda\in\Lambda}\chi_\lambda\mathring{P}_\lambda^{-1}(P-\mathring{P}_\lambda)(\chi_\lambda u)
		-\sum_{\lambda\in\Lambda}\chi_\lambda\mathring{P}_\lambda^{-1}([P,\chi_\lambda]u)
		-\chi_0Q_0([P,\chi_0]u)+\chi_0K_0(\chi_0u).
	\end{equation*}
	We write
	\begin{gather*}
		Su=\sum_{\lambda\in\Lambda}\chi_\lambda\mathring{P}_\lambda^{-1}(P-\mathring{P}_\lambda)(\chi_\lambda u),\qquad
		Tu=-\sum_{\lambda\in\Lambda}\chi_\lambda\mathring{P}_\lambda^{-1}([P,\chi_\lambda]u)-\chi_0Q_0([P,\chi_0]u),\\
		Ku=\chi_0K_0(\chi_0u),
	\end{gather*}
	so that $QP=I+S+T+K$.

	The third operator $K$ is obviously compact as an operator
	$C^{k,\alpha}_\delta\to C^{k,\alpha}_\delta$.
	The second operator $T$ is continuous as $C^{k,\alpha}_\delta\to C^{k+1,\alpha}_{\delta'}$,
	where $\delta'$ is taken so that $\delta<\delta'<\min(\delta+1,n+R_P)$.
	One can show that $C^{k+1,\alpha}_{\delta'}\hookrightarrow C^{k,\alpha}_\delta$ is a compact embedding
	(see \cite{Lee-06}*{Lemma 3.6 (d)} for the AH case), and hence $T$ is also compact
	as an operator $C^{k,\alpha}_\delta\to C^{k,\alpha}_\delta$.

	The first term, $Su$, has an estimate
	\begin{equation*}
		\norm{Su}_{C^{k,\alpha}_\delta}\le C\varepsilon^\nu\norm{u}_{C^{k,\alpha}_\delta}.
	\end{equation*}
	Therefore, if we take a sufficiently small $\varepsilon$, the operator norm of $S$ becomes less than $1$
	and therefore $I+S$ is invertible. We set $Q_1=(I+S)^{-1}Q$ and $K_1=(I+S)^{-1}(T+K)$ to get
	$Q_1Pu=u+K_1u$, where $K_1$ is compact.
	The other statement can be proved similarly.
\end{proof}

We need to specify the indicial radius of the operator \eqref{eq:linearized-Euler-Lagrange-bis}.

\begin{lem}
	\label{lem:indicial-roots}
	The indicial roots of the operator
	$P_S\colon\Gamma(X,\mathord{\wedge}^2_\mathrm{aH})\to\Gamma(X,\mathord{\wedge}^2_\mathrm{aH})$
	given by \eqref{eq:linearized-Euler-Lagrange-bis} are
	$n\pm\sqrt{n^2+2n+5}$ and $n\pm\sqrt{n^2+8}$,
	and hence, its indicial radius is $\sqrt{n^2+8}$.
\end{lem}

\begin{proof}
	Although the indicial roots are introduced in \cite{Matsumoto-preprint16} by using the polar coordinates
	associated with the representation $\mathbb{C}H^n=\PSU(n,1)/U(n)$,
	they are computable by expressing $P_S$ (on $\mathbb{C}H^n$) in the Siegel upper-half space coordinates.
	See the discussion following \cite{Matsumoto-preprint16}*{Proposition 1.4}.

	Let $X_0=\set{(z,w)\in\mathbb{C}^{n-1}\times\mathbb{C}|\Im w>\abs{z}^2}$
	and set $r=\Im w-\abs{z}^2$, $t=\Re w$, and $x=\sqrt{r/2}$.
	Then, the complex hyperbolic metric, normalized so that $\Ric=-(n+1)g$, is
	\begin{equation*}
		g=\frac{1}{2}\left(4\frac{dx^2}{x^2}+\frac{\theta^2}{x^4}
		+\frac{2}{x^2}\sum_{\alpha=1}^{n-1}dz^\alpha d\conjsmash{z}^\alpha\right),
	\end{equation*}
	where
	$\theta=\frac{1}{2}
	(dt+i\sum_{\alpha=1}^{n-1}(z^\alpha d\conjsmash{z}^\alpha-\conjsmash{z}^\alpha dz^\alpha))$.
	We define the frame $\set{\bm{Z}_0,\bm{Z}_1,\dots,\bm{Z}_{n-1}}$ of $T^{1,0}$ by
	\begin{equation*}
		\bm{Z}_0=\frac{1}{2}x\partial_x+2ix^2\partial_t,\qquad
		\bm{Z}_\alpha=x(\partial_{z^\alpha}+i\conjsmash{z}^\alpha\partial_t),\quad\alpha=1,\,\dots,\,n-1
	\end{equation*}
	so that
	\begin{equation*}
		\tensor{g}{_i_{\conj{j}}}=
		\begin{cases}
			1, & i=j=0,\\
			1/2, & i=j=1,\,\dots,\,n-1,\\
			0, & i\not=j.
		\end{cases}
	\end{equation*}
	Then, the Christoffel symbols are given as follows (cf.\ \cite{Matsumoto-16}*{Equation (5.2)}):
	\begin{alignat*}{4}
		\tensor{\Gamma}{^0_0_0}&=-1,&\qquad
		\tensor{\Gamma}{^0_{\conj{0}}_0}&=1,&\qquad
		\tensor{\Gamma}{^0_\alpha_0}&=0,&\qquad
		\tensor{\Gamma}{^0_{\conj{\alpha}}_0}&=0,\\
		\tensor{\Gamma}{^\gamma_0_0}&=0,&\qquad
		\tensor{\Gamma}{^\gamma_{\conj{0}}_0}&=0,&\qquad
		\tensor{\Gamma}{^\gamma_\alpha_0}&=-\tensor{\delta}{_\alpha^\gamma},&\qquad
		\tensor{\Gamma}{^\gamma_{\conj{\alpha}}_0}&=0,\\
		\tensor{\Gamma}{^0_0_\beta}&=0,&\qquad
		\tensor{\Gamma}{^0_{\conj{0}}_\beta}&=0,&\qquad
		\tensor{\Gamma}{^0_\alpha_\beta}&=0,&\qquad
		\tensor{\Gamma}{^0_{\conj{\alpha}}_\beta}&=\tfrac{1}{2}\tensor{\delta}{_\alpha_\beta},\\
		\tensor{\Gamma}{^\gamma_0_\beta}&=-\tfrac{1}{2}\tensor{\delta}{_\beta^\gamma},&\qquad
		\tensor{\Gamma}{^\gamma_{\conj{0}}_\beta}&=\tfrac{1}{2}\tensor{\delta}{_\beta^\gamma},&\qquad
		\tensor{\Gamma}{^\gamma_\alpha_\beta}&=0,&\qquad
		\tensor{\Gamma}{^\gamma_{\conj{\alpha}}_\beta}&=0.
	\end{alignat*}
	By using these, we can show that, if we omit the terms involving derivatives of $A$ in the directions of
	$t$ and $z^\alpha$, $\conjsmash{z}^\alpha$ (which do not contribute to the indicial polynomial),
	\begin{align*}
		\tensor{\nabla}{_0}\tensor{\nabla}{_{\conj{0}}}\tensor{A}{_0_\alpha}
		&=\frac{1}{4}(x\partial_x+1)(x\partial_x-3)\tensor{A}{_0_\alpha}
		+(\text{tangential derivatives}),\\
		\tensor{\nabla}{_\gamma}\tensor{\nabla}{_{\conj{\sigma}}}\tensor{A}{_0_\alpha}
		&=\frac{1}{4}\tensor{h}{_\gamma_{\conj{\sigma}}}(x\partial_x-3)\tensor{A}{_0_\alpha}
		-\tensor{h}{_[_\alpha_|_{\conj{\sigma}}}\tensor{A}{_0_|_\gamma_]}
		+(\text{tangential derivatives}),\\
		\tensor{\nabla}{_0}\tensor{\nabla}{_{\conj{0}}}\tensor{A}{_\alpha_\beta}
		&=\frac{1}{4}x\partial_x(x\partial_x-2)\tensor{A}{_\alpha_\beta}
		+(\text{tangential derivatives}),\\
		\tensor{\nabla}{_\gamma}\tensor{\nabla}{_{\conj{\sigma}}}\tensor{A}{_\alpha_\beta}
		&=-\frac{1}{4}\tensor{h}{_\gamma_{\conj{\sigma}}}(x\partial_x-2)\tensor{A}{_\alpha_\beta}
		+(\text{tangential derivatives}).
	\end{align*}
	Consequently,
	\begin{align*}
		\tensor{\nabla}{_k}\tensor{\nabla}{^k}\tensor{A}{_0_\alpha}
		&=\frac{1}{4}((x\partial_x)^2-2nx\partial_x+(2n-1))\tensor{A}{_0_\alpha}
		+(\text{tangential derivatives}),\\
		\tensor{\nabla}{_k}\tensor{\nabla}{^k}\tensor{A}{_\alpha_\beta}
		&=\frac{1}{4}((x\partial_x)^2-2nx\partial_x+(4n-4))\tensor{A}{_\alpha_\beta}
		+(\text{tangential derivatives})
	\end{align*}
	and hence,
	\begin{align*}
		P_S\tensor{A}{_0_\alpha}
		&=-\frac{1}{4}((x\partial_x)^2-2nx\partial_x-(2n+5))\tensor{A}{_0_\alpha}
		+(\text{tangential derivatives}),\\
		P_S\tensor{A}{_\alpha_\beta}
		&=-\frac{1}{4}((x\partial_x)^2-2nx\partial_x-8)\tensor{A}{_\alpha_\beta}
		+(\text{tangential derivatives}).
	\end{align*}
	This implies that the indicial roots are the roots of $s^2-2ns-(2n+5)$ and $s^2-2ns-8$; hence the claim.
\end{proof}

\section{Deformation of K\"ahler-Einstein structures}
\label{sec:proof-of-global-existence}

Again, let $\overline{X}$ be a compact smooth manifold-with-boundary of dimension $2n$
whose boundary $\bdry X$ is equipped with a contact distribution $H$.
Let $\alpha\in(0,1)$ be arbitrarily fixed.

In order to prove Theorem \ref{thm:global-existence},
we must recall the gauge-fixing technique employed by
Roth \cite{Roth-99-Thesis} and Biquard \cite{Biquard-00}.
We use the approach of \cite{Biquard-00}*{Section I.1.C} here (compare with \cite{Roth-99-Thesis}*{p.\ 31}).
The following condition is imposed in addition to $\Ric(\tilde{g})=-(n+1)\tilde{g}$,
where $g$ is some fixed Einstein ACH metric of class $C^{2,\alpha}_\delta$ with $\Ric(g)<0$
and $\delta_g$ is the divergence operator:
\begin{equation}
	\label{eq:Bianchi-gauge}
	\delta_g\tilde{g}+\frac{1}{2}d\tr_g\tilde{g}=0.
\end{equation}
This is known to be a slice condition for the action of diffeomorphisms
\cite{Biquard-00}*{Proposition I.4.6}, meaning that the mapping
\begin{equation}
	\label{eq:slice-condition-for-metrics}
	(\text{a neighborhood of $(0,g)$ in $C^{3,\alpha}_\delta(X,TX)\times
	\set{\tilde{g}\in g+C^{2,\alpha}_\delta
	\text{ satisfying \eqref{eq:Bianchi-gauge}}}$})
	\to g+C^{2,\alpha}_\delta
\end{equation}
defined by $(\xi,\tilde{g})\mapsto \Fl_\xi^*\tilde{g}$ is a homeomorphism near $(0,g)$.
In view of the diffeomorphism invariance of the Einstein equation,
it is reasonable to solve the equation under
\eqref{eq:Bianchi-gauge}, which is equivalent \cite{Biquard-00}*{Lemme I.1.4} to solving
\begin{equation}
	\label{eq:gauged-Einstein-equation}
	\Hat{E}
	=\Hat{E}_g(\tilde{g})
	:=\Ric(\tilde{g})+(n+1)\tilde{g}+\delta_{\tilde{g}}^*\left(\delta_g\tilde{g}+\frac{1}{2}d\tr_g\tilde{g}\right)
	=0.
\end{equation}
Once \eqref{eq:gauged-Einstein-equation} is solved ``locally uniquely,''
the discussion above implies that the original Einstein equation has a locally unique solution up to
the diffeomorphism action.

The linearization of $\Hat{E}$ under the change of $\tilde{g}$ is
the linearized gauged Einstein operator $P_{\Hat{E}}=\nabla_g^*\nabla_g-2\mathring{R}_g$ (half of it,
strictly speaking) mentioned in the introduction.
If $\ker_{(2)}P_{\Hat{E}}=0$, then Theorem \ref{thm:Fredholm} implies that $P_{\Hat{E}}$ is also an
isomorphism between some appropriate function spaces, which makes the implicit function theorem applicable.
This is the outline of the argument of \cite{Roth-99-Thesis} and \cite{Biquard-00},
and that of Graham and Lee \cite{Graham-Lee-91} in the AH case.

Now suppose that $(g,J)\in\tilde{\mathcal{M}}^{2,\alpha}_\delta$ is such that $\Ric(g)<0$.
The same argument for proving that \eqref{eq:slice-condition-for-metrics} is a local homeomorphism
can be used to show that the mapping
\begin{equation*}
	\begin{split}
		&(\text{a neighborhood of $(0,(g,J))$ in $C^{3,\alpha}_\delta(X,TX)\times
		\set{(\tilde{g},\tilde{J})\in (g,J)+C^{2,\alpha}_\delta
		\text{ satisfying \eqref{eq:Bianchi-gauge}}}$})\\
		&\to (g,J)+C^{2,\alpha}_\delta
	\end{split}
\end{equation*}
defined similarly is a local homeomorphism near $(0,(g,J))$.
Therefore, Theorem \ref{thm:global-existence} follows once we solve the system
\begin{equation}
	\label{eq:gauged-aKE-system}
	\Hat{E}=0,\qquad S=0
\end{equation}
locally uniquely.

In order to carry this out, we need a preliminary approximate solution to \eqref{eq:gauged-aKE-system}.

\begin{lem}
	\label{lem:first-approximate-solution}
	Let $(g,J)$ be any ACH almost Hermitian structure of class $C^{2,\alpha}_\delta$.
	Then $(g,J)$ is automatically
	an approximate solution of \eqref{eq:gauged-aKE-system} in the sense that
	\begin{equation*}
		\Hat{E}_g(g)=O(x^\delta)\qquad\text{and}\qquad
		S=O(x^\delta),
	\end{equation*}
	where $x$ is an arbitrary boundary defining function of $\overline{X}$.
\end{lem}

\begin{proof}
	The claim on $\Hat{E}$ is essentially proved in \cite{Biquard-00}*{Section I.4.B}
	(see also \cite{Roth-99-Thesis}*{p.\ 32}).
	What is considered there is a particular $g$, smooth up to the boundary,
	that is associated with an arbitrarily given
	$\gamma\in\mathcal{C}^{2,\alpha}_H$, and it is shown that $\Hat{E}=O(x)$.
	A general $g\in\mathcal{M}^{2,\alpha}_\delta$ is different from such a metric by an element of
	$C^{2,\alpha}_\delta$, and hence $\Hat{E}$ is $O(x^\delta)$.

	We can take a similar approach to show $S=O(x^\delta)$.
	Any $(g,J)\in\tilde{\mathcal{M}}^{2,\alpha}_\delta$ can be expressed
	as $g=g_{\theta,\gamma}+\sigma$ and $J=J_{\theta,\gamma}+\psi$,
	where $\sigma\in C^{2,\alpha}_\delta$ and $\psi\in C^{2,\alpha}_\delta$, for the model metric
	$g_{\theta,\gamma}$
	and the model almost complex structure $J_{\theta,\gamma}$.
	Recall that, if $\set{Z_\alpha}$ is a local frame of $T^{1,0}\bdry X$ and $T$ is the Reeb vector field for $\theta$,
	then
	\begin{equation*}
		\bm{Z}_0=\frac{1}{2}x\partial_x+ix^2T,\qquad
		\bm{Z}_\alpha=xZ_\alpha,\quad\alpha=1,\,\dots,\,n-1
	\end{equation*}
	span the holomorphic tangent bundle for $J_{\theta,\gamma}$.

	The connection coefficients $\Gamma$ and the torsion
	of the Ehresmann--Libermann connection $\nabla$ of
	$(g_{\theta,\gamma},J_{\theta,\gamma})$ with respect to $\set{\bm{Z}_0,\bm{Z}_\alpha}$ are computed below,
	and we will see that $S=O(x)$ for $(g_{\theta,\gamma},J_{\theta,\gamma})$.
	Then the difference between $\Gamma$ and $\Gamma^{(g,J)}$ is expressed in terms of $\sigma$, $\psi$ and
	the connection $\nabla$. By the way they are expressed, we can conclude
	that the difference between the tensor $S$ for those two almost Hermitian structures is $O(x^\delta)$,
	thereby showing $S=O(x^\delta)$ for $(g,J)$.

	Recall the Tanaka--Webster connection of the compatible almost CR structure $\gamma$
	in the sense defined in \cite{Matsumoto-14}*{Proposition 3.1}, which we write $\Hat{\nabla}$.
	The first structure equation is\footnote{Here we use the opposite sign
	and a different order of indices for the CR Nijenhuis tensor
	compared to \cite{Matsumoto-14}, for compatibility with our
	convention \eqref{eq:definition-of-Nijenhuis-tensor}.}
	\begin{equation*}
		d\theta^\gamma
		=\theta^\beta\wedge\tensor{\Hat{\omega}}{_\beta^\alpha}
		-\tensor{\Hat{A}}{_{\conj{\alpha}}^\gamma}\theta^{\conj{\alpha}}\wedge\theta
		+\frac{1}{2}\tensor{\Hat{N}}{^\gamma_{\conj{\alpha}}_{\conj{\beta}}}
		\theta^{\conj{\alpha}}\wedge\theta^{\conj{\beta}}.
	\end{equation*}
	We can check that the Ehresmann--Libermann connection of $(g_{\theta,\gamma},J_{\theta,\gamma})$ is
	given by the following coefficients with respect to the frame
	$\set{\bm{Z}_0,\bm{Z}_\alpha,\bm{Z}_{\conj{0}},\bm{Z}_{\conj{\alpha}}}$:
	\begin{alignat*}{4}
		\tensor{\Gamma}{^0_0_0}&=-1,&\qquad
		\tensor{\Gamma}{^0_{\conj{0}}_0}&=1,&\qquad
		\tensor{\Gamma}{^0_\alpha_0}&=0,&\qquad
		\tensor{\Gamma}{^0_{\conj{\alpha}}_0}&=0,\\
		\tensor{\Gamma}{^\gamma_0_0}&=0,&\qquad
		\tensor{\Gamma}{^\gamma_{\conj{0}}_0}&=0,&\qquad
		\tensor{\Gamma}{^\gamma_\alpha_0}&=-\tensor{\delta}{_\alpha^\gamma},&\qquad
		\tensor{\Gamma}{^\gamma_{\conj{\alpha}}_0}&=ix^2\tensor{\Hat{A}}{_{\conj{\alpha}}^\gamma},\\
		\tensor{\Gamma}{^0_0_\beta}&=0,&\qquad
		\tensor{\Gamma}{^0_{\conj{0}}_\beta}&=0,&\qquad
		\tensor{\Gamma}{^0_\alpha_\beta}&=ix^2\tensor{\Hat{A}}{_\alpha_\beta},&\qquad
		\tensor{\Gamma}{^0_{\conj{\alpha}}_\beta}&=\tfrac{1}{2}\tensor{h}{_\alpha_{\conj{\beta}}},\\
		\tensor{\Gamma}{^\gamma_0_\beta}
		&=ix^2\tensor{\Hat{\Gamma}}{^\gamma_0_\beta}-\tfrac{1}{2}\tensor{\delta}{_\beta^\gamma},&\qquad
		\tensor{\Gamma}{^\gamma_{\conj{0}}_\beta}
		&=-ix^2\tensor{\Hat{\Gamma}}{^\gamma_0_\beta}+\tfrac{1}{2}\tensor{\delta}{_\beta^\gamma},&\qquad
		\tensor{\Gamma}{^\gamma_\alpha_\beta}&=x\tensor{\Hat{\Gamma}}{^\gamma_\alpha_\beta},&\qquad
		\tensor{\Gamma}{^\gamma_{\conj{\alpha}}_\beta}&=x\tensor{\Hat{\Gamma}}{^\gamma_{\conj{\alpha}}_\beta}.
	\end{alignat*}
	Therefore, we obtain
	$\tensor{N}{^{\conj{\gamma}}_0_\beta}=ix^2\tensor{\Hat{A}}{_\beta^{\conj{\gamma}}}$,
	$\tensor{N}{^{\conj{\gamma}}_\alpha_\beta}=x\tensor{\Hat{N}}{^{\conj{\gamma}}_\alpha_\beta}$,
	and $T=0$. This implies that $S=O(x)$ for $(g_{\theta,\gamma},J_{\theta,\gamma})$.
\end{proof}

We introduce here one more technical tool, which is the projection to the space of almost complex structures.

Let $(V,b)$ be a real $2n$-dimensional vector space with inner product,
and $\mathcal{J}_{(V,b)}$ the set of linear complex structures of $V$ compatible with $b$.
The orthogonal group $O(V,b)$ acts properly on $\End(V)$ by conjugation, and hence $\mathcal{J}_{(V,b)}$ is,
being an orbit of the action, a closed submanifold;
the tangent spaces of $\mathcal{J}_{(V,b)}$ are given by the infinitesimal action of $O(V,b)$ at each point.
The space $\End(V)$ carries the distance defined by the Hilbert--Schmidt inner product, and in a neighborhood of
$\mathcal{J}_{(V,b)}$ in $\End(V)$,
the nearest point projection onto $\mathcal{J}_{(V,b)}$ is well-defined and smooth.

The definition of the above nearest point projection can be extended to the manifold setting.
Let $(X,g)$ be a Riemannian manifold, and suppose that a smooth section $\varphi$ of $\End(TX)$
pointwisely sufficiently close to $\mathcal{J}_{(T_xX,g_x)}$ is given.
Then the projection of $\varphi$ can be taken pointwisely,
and one obtains a smooth almost complex structure on $X$ compatible with $g$,
which we write $\pi_g(\varphi)$.

\begin{proof}[Proof of Theorem \ref{thm:global-existence}]
	Take any family $(g_\gamma,J_\gamma)$ of ACH almost Hermitian structures of class $C^{2,\alpha}_\delta$
	smooth in $\gamma\in\mathcal{U}$ such that $(g_{\gamma_0},J_{\gamma_0})=(g,J)$.
	This is possible as follows for instance.
	We express $(g,J)$ as $g=g_{\theta,\gamma_0}+\sigma$ and $J=J_{\theta,\gamma_0}+\psi$
	with respect to an admissible collar neighborhood diffeomorphism $\Phi$.
	Then we define $g_\gamma=g_{\theta,\gamma}+\sigma$ near $\bdry X$, and extend it to the whole $X$
	by a partition of unity, so that the extended $g_\gamma$ is $C^{2,\alpha}$-close to $g$.
	Likewise, we define $\varphi_\gamma=J_{\theta,\gamma}+\psi$ near $\bdry X$ and extend by a partition of
	unity to $X$. Then we use the nearest point projection with respect to $g_\gamma$ to define
	$J_\gamma=\pi_{g_\gamma}(\varphi_\gamma)$.

	The family $(g_\gamma,J_\gamma)$ uniformly satisfy $\Hat{E}=O(x^\delta)$ and $S=O(x^\delta)$
	by Lemma \ref{lem:first-approximate-solution}.
	Note that $n-R_{P_S}<\delta<n+R_{P_S}$ by Lemma \ref{lem:indicial-roots},
	and also that $n-R_{P_{\Hat{E}}}<\delta<n+R_{P_{\Hat{E}}}$ because $R_{P_{\Hat{E}}}=n$
	(see the discussion following \cite{Matsumoto-preprint16}*{Proposition 1.4}).

	Let $\mathord{\wedge}^2_{\mathrm{aH}}$ be the bundle of 2-forms that are anti-Hermitian
	with respect to the almost complex structure $J$.
	We define the mapping
	\begin{equation*}
		Q\colon\mathcal{U}\times\mathcal{V}_1\times\mathcal{V}_2
		\to C^{0,\alpha}_\delta(X,\Sym^2T^*)\oplus C^{0,\alpha}_\delta(X,\mathord{\wedge}^2_\mathrm{aH}),
	\end{equation*}
	where $\mathcal{V}_1$ is a small neighborhood of $0\in C^{2,\alpha}_\delta(X,\Sym^2T^*)$ and
	$\mathcal{V}_2$ is a small neighborhood of $0\in C^{2,\alpha}_\delta(X,\mathord{\wedge}^2_\mathrm{aH})$,
	as follows ($\mathcal{U}$, $\mathcal{V}_1$, and $\mathcal{V}_2$ will be, when needed, made smaller
	without notice in the sequel).
	For $(\gamma,\tau,\chi)\in\mathcal{U}\times\mathcal{V}_1\times\mathcal{V}_2$,
	we first define $g_{(\gamma,\tau)}=g_\gamma+\tau$.
	Then we define $J_{(\gamma,\tau,\chi)}=\pi_{g_{(\gamma,\tau)}}(J_\gamma+\chi)$,
	for which $(g_{(\gamma,\tau)},J_{(\gamma,\tau,\chi)})$ is again an ACH almost Hermitian structure
	of class $C^{2,\alpha}_\delta$.
	Finally, we set
	\begin{equation*}
		Q(\gamma,\tau,\chi)=(\Hat{E}_{g_\gamma}(g_{(\gamma,\tau)}),S_{g_{(\gamma,\tau)}}(J_{(\gamma,\tau,\chi)})).
	\end{equation*}
	The mapping $Q$ is smooth in $\gamma$, $\tau$, and $\chi$.

	We now use Theorem \ref{thm:Fredholm}.
	The linearization of $\Hat{E}_{g_\gamma}(g_{(\gamma,\tau)})$ at $(\gamma_0,0,0)$ with respect to
	the second parameter $\tau$ is $P_{\Hat{E}}$,
	and it is an isomorphism as a mapping $C^{2,\alpha}_\delta\to C^{0,\alpha}_\delta$,
	because $\ker_{(2)}P_{\Hat{E}}=0$ by the assumption.
	Likewise, the linearization of $S_{g_{(\gamma,\tau)}}(J_{(\gamma,\tau,\chi)})$ at $(\gamma_0,0,0)$
	with respect to the third parameter $\chi$ is $P_S$, and it is an isomorphism as a mapping
	$C^{2,\alpha}_\delta\to C^{0,\alpha}_\delta$,
	because $\ker_{(2)}P_S=0$, which is obvious from \eqref{eq:linearized-Euler-Lagrange-bis}.
	Consequently, the linearization of $Q$ at $(\gamma_0,0,0)$ with respect to
	the second and the third parameters is in a form
	\begin{equation*}
		\begin{pmatrix}
			P_{\Hat{E}} & 0 \\
			* & P_S
		\end{pmatrix},
	\end{equation*}
	and this is an isomorphism as a mapping
	$C^{2,\alpha}_\delta\oplus C^{2,\alpha}_\delta\to C^{0,\alpha}_\delta\oplus C^{0,\alpha}_\delta$.
	By the implicit function theorem, if $\mathcal{U}$ is sufficiently small,
	for each $\gamma\in\mathcal{U}$ there exists only one $(\tau,\chi)$
	in an appropriate neighborhood of $(0,0)\in\mathcal{V}_1\times\mathcal{V}_2$
	for which $Q(\gamma,\tau,\chi)=0$ is satisfied,
	or equivalently, $\Hat{E}_{g_\gamma}(g_{(\gamma,\tau)})=0$ and
	$S_{g_{(\gamma,\tau)}}(J_{(\gamma,\tau,\chi)})=0$.
\end{proof}

\section{Approximate solutions of higher order}
\label{sec:approximate-solutions}

We turn to the proof of Theorem \ref{thm:formal-existence}.
Some part of the theorem is already shown by the author in \cite{Matsumoto-14}.
Specifically, we use the following version, proved in \cite{Matsumoto-13-Thesis}
(see also \cite{Matsumoto-16}*{Theorem 2.5}).

\begin{thm}[Matsumoto \cite{Matsumoto-13-Thesis}*{Theorem 2.1}]
	\label{thm:approximate-solution-of-Einstein-equation}
	Let $\overline{X}$ be a manifold-with-boundary whose boundary is equipped with
	a contact distribution $H$, and let $\gamma\in\mathcal{C}_H$.
	Then, there exists an ACH metric $g$ on $X$, with conformal infinity $\gamma$, that is smooth up to the boundary
	satisfying
	\begin{equation}
		\Ric(g)=-(n+1)g+O(x^{2n}).
	\end{equation}
	Up to the action of diffeomorphisms of $\overline{X}$ that restricts to the identity on the boundary,
	such an ACH metric is unique modulo $O(x^{2n})$ ambiguity.
\end{thm}

It is known that, by identifying a neighborhood of $\bdry X$ in $\overline{X}$
with $\bdry X\times[0,\varepsilon)_x$
by an appropriately chosen admissible collar neighborhood diffeomorphism $\Phi$,
we can further normalize $g$ in such a way that $\partial/\partial x$ is orthogonal to the level sets of $x$
(Guillarmou--S\'a Barreto \cite{Guillarmou-SaBarreto-08}*{Section 3.2}).
Under this additional normalization condition, the metric $g$ is unique modulo $O(x^{2n})$ ambiguities,
and the proof of Theorem \ref{thm:approximate-solution-of-Einstein-equation} shows that
the expansion of $g$ in $x$ up to the $(2n-1)$-st order
has a local formula in terms of the Tanaka--Webster connection of $(\gamma,\theta)$.
In this sense, the expansion of $g$ is locally determined by the geometry of the boundary.

Therefore, in order to show Theorem \ref{thm:formal-existence}, it suffices to prove the following.
Let $\bm{Z}_\alpha$ and $\bm{Z}_{\conj{\alpha}}$ ($\alpha=1$, $\dotsc$, $n-1$) be as in the beginning of
Section \ref{subsec:ACH-almost-Hermitian-structures}.

\begin{prop}
	\label{prop:forma-existence-of-critical-almost-complex-structure}
	Let $\overline{X}$ be a compact manifold-with-boundary of dimension $2n$
	whose boundary is equipped with a contact distribution $H$,
	and $g$ be an ACH metric that is smooth up to the boundary
	satisfying $\Ric(g)=-(n+1)g+O(x^{2n})$, whose conformal infinity is denoted by $\gamma\in\mathcal{C}_H$.
	Let an open neighborhood of $\bdry X$ and $\bdry X\times[0,\varepsilon)_x$ be identified by
	an admissible collar neighborhood diffeomorphism $\Phi$ with which $g$ is normalized
	in the sense described above.
	Then, in a neighborhood of $\bdry X$,
	there exists an almost complex structure $J$, for which $(g,J)$ is an ACH almost Hermitian structure
	with conformal infinity $\gamma$
	that is smooth up to the boundary satisfying
	\begin{equation*}
		S=O(x^{2n}).
	\end{equation*}
	Moreover, such $J$ is uniquely determined modulo $O(x^{2n})$ ambiguity,
	and the components of $J$ with respect to the frame
	$\set{x\partial_x,x^2T,\bm{Z}_1,\dots,\bm{Z}_{n-1},\bm{Z}_{\conj{1}},\dots,\bm{Z}_{\conj{n-1}}}$,
	if expanded in $x$, have coefficients up to the $(2n-1)$-st order given by certain
	universal expressions in terms of the Tanaka--Webster local invariants of $(\gamma,\theta)$.
\end{prop}

In order to prove this proposition, we will use the pointwise nearest point projection $\pi_g$ to
the submanifold of compatible linear complex structures introduced in the previous section.
The following observation is useful in the proof.

\begin{lem}
	\label{lem:distance-and-non-almost-Hermitianity}
	Let $(V,b)$ be a real $2n$-dimensional vector space with inner product,
	and $\mathcal{J}$ the set of linear complex structures of $V$ compatible with $b$.
	Let $f\colon \End(V)\to\End(V)$ and $g\colon \End(V)\to\Sym^2V^*$ be defined by
	\begin{equation*}
		f(\varphi)=\varphi^2+\id,\qquad
		g(\varphi)=b(\varphi\cdot,\varphi\cdot)-b(\cdot,\cdot),
	\end{equation*}
	with which we have $\mathcal{J}=(f,g)^{-1}(0,0)$.
	Then, in a neighborhood $\mathcal{U}$ of $\mathcal{J}$ in $\End(V)$,
	there exists a constant $C>0$ for which
	\begin{equation*}
		\dist(\varphi,\mathcal{J})\le C(\abs{f(\varphi)}+\abs{g(\varphi)}).
	\end{equation*}
\end{lem}

\begin{proof}
	We may assume that $V=\mathbb{R}^{2n}$ and $b$ is the standard inner product.
	Then we can write $f(\varphi)=\varphi^2+I$ and $g(\varphi)=\transpose{\varphi}\varphi-I$,
	where $I$ is the identity matrix.
	Moreover, since $\mathcal{J}$ is compact, it suffices to argue locally.

	First, let $F=(f,g)\colon\End(V)\to\End(V)\oplus\Sym^2V^*$,
	and we want to show that the kernel of $(dF)_J$ at each $J\in\mathcal{J}$ agrees with
	the tangent space $T_J\mathcal{J}$.
	The kernel of $(dF)_J$ is given by
	\begin{equation*}
		J\varphi+\varphi J=0,\qquad
		\transpose{J}\varphi+\transpose{\varphi}J=0,
	\end{equation*}
	which is equivalent to that $J\varphi$ being anti-Hermitian and skew-symmetric (note that $\transpose{J}=-J$).
	On the other hand, the tangent space of $\mathcal{J}$ is given by
	the infinitesimal action of the orthogonal group,
	which implies $T_J\mathcal{J}=\set{XJ-JX|X\in \mathfrak{o}(2n)}$.
	If we set $\varphi=XJ-JX$, then $J\varphi=X+JXJ$, which is the anti-Hermitian part of $X$.
	Therefore, $T_J\mathcal{J}$ also consists of all those $\varphi$
	for which $J\varphi$ is a skew-symmetric anti-Hermitian matrix.

	We fix a point $J\in\mathcal{J}$.
	Then, near $J$, we can take a smooth local coordinate system $\xi=(\xi',\xi'')$ of $\End(V)$
	for which $\mathcal{J}$ is defined by $\xi'=0$.
	Using these coordinates, and by identifying $\End(V)\oplus\Sym^2V^*$ with $\mathbb{R}^N$, we can write
	\begin{equation*}
		F(\xi',\xi'')
		=F(0,\xi'')+A_{\xi''}\xi'+O(\abs{\xi'}^2)
		=A_{\xi''}\xi'+O(\abs{\xi'}^2),
	\end{equation*}
	with a smooth family $A_{\xi''}$ of matrices of full rank.
	Therefore we have $\abs{F(\xi',\xi'')}\ge c\abs{\xi'}$ for some $c>0$,
	which implies that $\abs{F(\varphi)}\ge c'\dist(\varphi,\mathcal{J})$ near $J$ for some $c'>0$.
\end{proof}

\begin{proof}[Proof of Proposition \ref{prop:forma-existence-of-critical-almost-complex-structure}]
	For notational simplicity, let $(g_0,J_0)=(g_{\theta,\gamma},J_{\theta,\gamma})$.
	In a sufficiently small neighborhood of $\bdry X$, $\pi_g(J_0)$ is defined.
	We set $J_1=\pi_g(J_0)$; then $S=O(x)$.
	Actually, by Lemma \ref{lem:first-approximate-solution}, any almost complex structure $J$ for which
	$(g,J)$ is an ACH almost Hermitian structure that is smooth up to the boundary satisfies $S=O(x)$.

	Then, we shall argue inductively.
	Supposing that $(g,J_l)$ is an ACH almost Hermitian structure with conformal infinity $\gamma$
	that is smooth up to the boundary for which $S=O(x^l)$ is satisfied, where $1\le l\le 2n-1$,
	we construct $J_{l+1}$ satisfying $S=O(x^{l+1})$.
	In addition, we will show that the coefficients of the components of $J_{l+1}$ with respect to
	$\set{x\partial_x,x^2T,\bm{Z}_\alpha,\bm{Z}_{\conj{\alpha}}}$ expanded in $x$
	are, up to the $l$-th order, given in terms of the Tanaka--Webster local invariants.

	Let $J_l$ be given for some $l\ge 1$.
	We truncate the Taylor expansions of the components of $J_l$
	so that all the components becomes polynomials in $x$ of degree (at most) $l-1$.
	The resulting section of $\End(TX)$, defined only near $\bdry X$, is denoted by $J^\mathrm{trc}_l$.
	Let us write
	\begin{equation*}
		J_l=J^\mathrm{trc}_l+x^l\varphi+O(x^{l+1}),
	\end{equation*}
	where $\varphi$ is a section of $\End(TX)$ whose components with respect to
	$\set{x\partial_x,x^2T,\bm{Z}_\alpha,\bm{Z}_{\conj{\alpha}}}$ are constant in $x$,
	and similarly, we write $J_{l+1}$, which is to be determined, as
	\begin{equation*}
		J_{l+1}=J^\mathrm{trc}_l+x^l\psi+O(x^{l+1}).
	\end{equation*}
	Moreover, we can decompose each of $\varphi$ and $\psi$ by
	symmetry/skew-symmetry with respect to $g_0$ and
	Hermitianity/anti-Hermitianity with respect to $J_0$.
	We symbolize these decompositions as $S$/$\wedge$ and H/aH, respectively,
	and we express $\varphi$ and $\psi$ as
	$\varphi=\varphi_{S_\mathrm{H}}+\varphi_{S_\mathrm{aH}}+\varphi_{\wedge_\mathrm{H}}+\varphi_{\wedge_\mathrm{aH}}$
	and
	$\psi=\psi_{S_\mathrm{H}}+\psi_{S_\mathrm{aH}}+\psi_{\wedge_\mathrm{H}}+\psi_{\wedge_\mathrm{aH}}$.

	We first want to show that the almost Hermitian requirements
	$J_{l+1}^2=-I$ and $g(J_{l+1}\cdot,J_{l+1}\cdot)=g(\cdot,\cdot)$
	imply that $\psi_{S_\mathrm{H}}$, $\psi_{S_\mathrm{aH}}$, and $\psi_{\wedge_\mathrm{H}}$
	must agree with the corresponding components of $\varphi$.
	To see this, we write $J_{l+1}=J_l+x^l\chi+O(x^{l+1})$, i.e., $\chi=\psi-\varphi$. Then
	\begin{equation*}
		J_{l+1}^2=J_l^2+x^l(J_0\chi+\chi J_0)+O(x^{l+1}),
	\end{equation*}
	and this means that $\chi_{S_\mathrm{H}}$ and $\chi_{\wedge_\mathrm{H}}$ must vanish.
	Likewise, we write
	\begin{equation*}
		g(J_{l+1}V,J_{l+1}W)=g(J_lV,J_lW)+x^l(g_0(J_0V,\chi W)+g_0(\chi V,J_0W))+O(x^{l+1}),
	\end{equation*}
	and this implies that $\chi_{S_\mathrm{aH}}$ and $\chi_{\wedge_\mathrm{H}}$ must vanish.
	Note that these computations also show that the contribution of $\chi_{\wedge_\mathrm{H}}$ to
	$J_{l+1}^2$ and $g(J_{l+1}\cdot,J_{l+1}\cdot)$ is only $O(x^{l+1})$.
	By this, we can even say that
	$J^\mathrm{trc}_l+x^l(\varphi_{S_\mathrm{H}}+\varphi_{S_\mathrm{aH}}+\varphi_{\wedge_\mathrm{H}}
	+\psi_{\wedge_\mathrm{H}})$ satisfies the almost Hermitian requirements modulo $O(x^{l+1})$,
	no matter what $\psi_{\wedge_\mathrm{H}}$ is.

	Similar computations using $J_l=J^\mathrm{trc}_l+x^l\varphi+O(x^{l+1})$ show that
	$\varphi_{S_\mathrm{H}}$, $\varphi_{S_\mathrm{aH}}$, and $\varphi_{\wedge_\mathrm{H}}$ are given in terms of
	the Tanaka--Webster local invariants as follows. Since
	\begin{equation*}
		J_l^2=(J^\mathrm{trc}_l)^2+x^l(J_0\varphi+\varphi J_0)+O(x^{l+1}),
	\end{equation*}
	$-(J_0\varphi+\varphi J_0)$ should agree with the $x^l$-coefficient of
	$(J^\mathrm{trc}_l)^2$, which is obviously written in terms of the Tanaka--Webster invariants.
	Hence, $\varphi_{S_\mathrm{H}}$ and $\varphi_{\wedge_\mathrm{H}}$ are given by such invariants.
	The other equality
	\begin{equation*}
		g(J_lV,J_lW)
		=g(J^\mathrm{trc}_lV,J^\mathrm{trc}_lW)+x^l(g_0(J_0V,\varphi W)+g_0(\varphi V,J_0W))+O(x^{l+1})
	\end{equation*}
	implies that $\varphi_{\wedge_\mathrm{H}}$ is given like so.

	Now we want to introduce a skew-symmetric anti-Hermitian part in a unique way
	by the requirement $S=O(x^{l+1})$.
	We set
	\begin{equation*}
		J'_l=\pi_g(J^\mathrm{trc}_l+x^l(\varphi_{S_\mathrm{H}}+\varphi_{S_\mathrm{aH}}+\varphi_{\wedge_\mathrm{H}}))
		\quad\text{and}\quad
		J_{l+1}
		=\pi_g(J^\mathrm{trc}_l+x^l(\varphi_{S_\mathrm{H}}+\varphi_{S_\mathrm{aH}}+\varphi_{\wedge_\mathrm{H}}+A)),
	\end{equation*}
	where $A$ is skew-symmetric anti-Hermitian with respect to $(g_0,J_0)$.
	In view of Lemma \ref{lem:distance-and-non-almost-Hermitianity},
	the terms introduced by applying $\pi_g$ here are both $O(x^{l+1})$, and hence, we have
	$J_{l+1}=J'_l+x^lA+O(x^{l+1})$. Then,
	the computation in the proof of Lemma \ref{lem:indicial-roots} shows that
	the tensor $S$ for $J_{l+1}$ is given as follows (in terms of $S'$, which is $S$ for $J'_l$):
	\begin{align*}
		\tensor{S}{_0_\alpha}
		&=S'_{0\alpha}
		-\frac{1}{4}(l^2-2nl-(2n+5))x^l\tensor{A}{_0_\alpha}+O(x^{l+1}),\\
		\tensor{S}{_\alpha_\beta}
		&=S'_{\alpha\beta}
		-\frac{1}{4}(l^2-2nl-8)x^l\tensor{A}{_\alpha_\beta}+O(x^{l+1}).
	\end{align*}
	Since $l^2-2nl-(2n+5)$ and $l^2-2nl-8$ are never zero, $A$ is uniquely determined by
	the requirement that $S=O(x^{l+1})$.
	The construction of $J'_l$ implies that
	the expansion of $S'$ is expressed in terms of the Tanaka--Webster local invariants up to $l$-th order,
	and hence so is $A$.
\end{proof}

\section{A discussion on general second-order functionals}
\label{sec:possible-functionals}

We here establish a partial characterization of our functional $\mathcal{E}$
to give some justification for our choice.
The most general reasonable choice of functionals is given by
the integral of a linear combination of complete contractions of tensor products of the form
\begin{equation*}
	(\tensor{R}{_i_{\conj{j}}_k_{\conj{l}}})^{\otimes m_1}
	\otimes (\tensor{R}{_i_{\conj{j}}_k_l})^{\otimes m_2}
	\otimes (\tensor{R}{_i_{\conj{j}}_{\conj{k}}_{\conj{l}}})^{\otimes m_3}
	\otimes (\tensor{N}{_i_j_k})^{\otimes m_4}\otimes(\tensor{N}{_{\conj{i}}_{\conj{j}}_{\conj{k}}})^{\otimes m_5}
	\otimes (\tensor{T}{_{\conj{i}}_j_k})^{\otimes m_6}\otimes(\tensor{T}{_i_{\conj{j}}_{\conj{k}}})^{\otimes m_7}.
\end{equation*}
If we require that the Euler--Lagrange equation is a second-order partial differential equation,
then the integrand must be a linear combination of
\begin{equation*}
	R=\tensor{R}{_i^i_j^j},\qquad
	\tensor{R}{_i^j_j^i},\qquad
	\abs{N}^2,\qquad
	\tensor{N}{_i_j_k}\tensor{N}{^j^i^k},\qquad
	\abs{T}^2,\qquad
	\abs{\tau}^2,\qquad
	\delta=\tensor{\nabla}{^i}\tensor{\tau}{_i}.
\end{equation*}
Because of \eqref{eq:first-Bianchi-21}, we have $\tensor{R}{_i^j_j^i}=R+\delta-\tensor{N}{_i_j_k}\tensor{N}{^j^i^k}$,
and \eqref{eq:divergence-formula} implies that the integral of $\delta$ equals that of $-\abs{\tau}^2$.
Hence, we may exclude $\tensor{R}{_i^j_j^i}$ and $\delta$ from the list.
Moreover, because the difference between the Levi-Civita and the Ehresmann--Libermann connections is
given in terms of $N$ and $T$,
the Riemannian scalar curvature of $g$ equals $2R$ plus a linear combination of
$\abs{N}^2$, $\tensor{N}{_i_j_k}\tensor{N}{^j^i^k}$, $\abs{T}^2$, $\abs{\tau}^2$.
We can also reasonably omit $R$, because its integral is invariant under a change of $J$.

Rather than $\abs{N}^2$ and $\tensor{N}{_i_j_k}\tensor{N}{^j^i^k}$,
we prefer to use the squared norms of
$\tensor{(\Nsym)}{_i_j_k}=\tensor{N}{_(_i_j_)_k}$ and $\tensor{(\Nskew)}{_i_j_k}=\tensor{N}{_[_i_j_]_k}$,
which is possible by the relations
$\abs{N}^2=\abs{\Nsym}^2+\abs{\Nskew}^2$ and
$\tensor{N}{_i_j_k}\tensor{N}{^j^i^k}=\abs{\Nsym}^2-\abs{\Nskew}^2$. Thus, the list becomes
\begin{equation*}
	\abs{\Nsym}^2,\qquad
	\abs{\Nskew}^2,\qquad
	\abs{T}^2,\qquad
	\abs{\tau}^2.
\end{equation*}
We call the integral of any linear combination of these four quantities
a \emph{second-order functional} of almost complex structures compatible with a given Riemannian metric $g$.
Let
\begin{equation*}
	\mathcal{E}_{(a,b,c,d)}
	=\int (a\abs{\Nsym}^2+b\abs{\Nskew}^2+c\abs{T}^2+d\abs{\tau}^2)dV_g.
\end{equation*}
Then, the functional $\mathcal{E}$ is $\mathcal{E}_{(1,1,0,1/2)}$.

Our partial characterization of $\mathcal{E}$ is the following.

\begin{prop}
	\label{prop:characterization}
	A second-order functional $\mathcal{E}_{(a,b,c,d)}$ has Euler--Lagrange equation whose linearization equals
	$\frac{1}{2}\Delta_{\conj{\partial}}$ if and only if
	\begin{equation*}
		a=1+s,\qquad
		b=1-3s,\qquad
		c=s,\qquad
		d=\frac{1}{2}-2s
	\end{equation*}
	for some $s\in\mathbb{R}$.
\end{prop}

To show Proposition \ref{prop:characterization}, let us write
\begin{equation*}
	\mathcal{E}^\bullet=\int\abs{\bullet}^2dV_g
\end{equation*}
for $\bullet=\Nsym$, $\Nskew$, $T$, $\tau$, and
\begin{equation*}
	\left.\frac{d}{dt}\mathcal{E}^\bullet[J_t]\right|_{t=0}
	=\int((\dot{\mathcal{E}}^\bullet)^{ij}\tensor{A}{_i_j}
	+(\dot{\mathcal{E}}^\bullet)^{\conj{i}\conj{j}}\tensor{A}{_{\conj{i}}_{\conj{j}}})dV_g
\end{equation*}
as in Section \ref{sec:Euler-Lagrange-equation}.
Then, a computation akin to the proof of Proposition \ref{prop:termwise-Euler-Lagrange}
gives the following formulae
(actually \eqref{eq:Euler-Lagrange-tau-anti-Hermitian} is already given there).

\begin{lem}
	Under the notation above,
	\begin{subequations}
	\begin{align}
		\label{eq:Euler-Lagrange-Nijenhuis-sym-anti-Hermitian}
		\dot{\mathcal{E}}^\Nsym_{ij}
		&=i\left(-\frac{1}{4}(\tensor{\nabla}{^k}+\tensor{\tau}{^k})\tensor{N}{_k_i_j}
			+\frac{1}{4}(\tensor{\nabla}{^k}+\tensor{\tau}{^k})\tensor{(\Nskew)}{_i_j_k}
			+\frac{1}{2}\tensor{(\Nsym)}{_[_i_|_k_l}\tensor{T}{_|_j_]^k^l}\right),\\
		\label{eq:Euler-Lagrange-Nijenhuis-skew-anti-Hermitian}
		\dot{\mathcal{E}}^\Nskew_{ij}
			&=i\left(\frac{1}{4}(\tensor{\nabla}{^k}+\tensor{\tau}{^k})\tensor{N}{_k_i_j}
			+\frac{3}{4}(\tensor{\nabla}{^k}+\tensor{\tau}{^k})\tensor{(\Nskew)}{_i_j_k}
			+\frac{1}{2}\tensor{(\Nskew)}{_[_i_|_k_l}\tensor{T}{_|_j_]^k^l}\right),\\
		\label{eq:Euler-Lagrange-T-anti-Hermitian}
		\dot{\mathcal{E}}^T_{ij}
			&=i\left(-(\tensor{\nabla}{^{\conj{k}}}+\tensor{\tau}{^{\conj{k}}})\tensor{T}{_{\conj{k}}_i_j}
			+\tensor{N}{_k_l_[_i}\tensor{T}{_j_]^k^l}
			-\frac{1}{2}\tensor{N}{_[_i_|_k_l}\tensor{T}{_|_j_]^k^l}\right),\\
		\label{eq:Euler-Lagrange-tau-anti-Hermitian}
		\dot{\mathcal{E}}^\tau_{ij}
			&=i\left(\tensor{\nabla}{_[_i}\tensor{\tau}{_j_]}
			-\frac{1}{2}\tensor{N}{_k_i_j}\tensor{\tau}{^k}
			+\frac{1}{2}\tensor{T}{^k_i_j}\tensor{\tau}{_k}\right).
	\end{align}
	\end{subequations}
\end{lem}

Next, recall from \eqref{eq:variation-Nijenhuis} and \eqref{eq:variation-torsion} that,
under the K\"ahler-Einstein assumption,
\begin{alignat*}{3}
	\tensor{\dot{N}}{^k_i_j}&=0,&\qquad
	\tensor{\dot{N}}{^k_i_{\conj{j}}}&=0,&\qquad
	\tensor{\dot{N}}{^k_{\conj{i}}_{\conj{j}}}
	&=-i\tensor{\nabla}{_[_{\conj{i}}}\tensor{A}{^k_{\conj{j}}_]},\\
	\tensor{\dot{T}}{^k_i_j}
	&=-i\tensor{\nabla}{^k}\tensor{A}{_i_j},&\qquad
	\tensor{\dot{T}}{^k_i_{\conj{j}}}&=0,&\qquad
	\tensor{\dot{T}}{^k_{\conj{i}}_{\conj{j}}}&=0.
\end{alignat*}
As a consequence of the first line, we also have
\begin{alignat*}{3}
	\tensor{(\dotNsym)}{^k_i_j}&=0,&\qquad
	\tensor{(\dotNsym)}{^k_i_{\conj{j}}}&=0,&\qquad
	\tensor{(\dotNsym)}{^k_{\conj{i}}_{\conj{j}}}
	&=-\frac{i}{4}(\tensor{\nabla}{^k}\tensor{A}{_{\conj{i}}_{\conj{j}}}
	+\tensor{\nabla}{_{\conj{i}}}\tensor{A}{^k_{\conj{j}}}),\\
	\tensor{(\dotNskew)}{^k_i_j}&=0,&\qquad
	\tensor{(\dotNskew)}{^k_i_{\conj{j}}}&=0,&\qquad
	\tensor{(\dotNskew)}{^k_{\conj{i}}_{\conj{j}}}
	&=\frac{i}{4}(\tensor{\nabla}{^k}\tensor{A}{_{\conj{i}}_{\conj{j}}}
	-\tensor{\nabla}{_{\conj{i}}}\tensor{A}{^k_{\conj{j}}}
	+2\tensor{\nabla}{_{\conj{j}}}\tensor{A}{^k_{\conj{i}}}).
\end{alignat*}
Using these formulae, we can now compute the linearizations of $\dot{\mathcal{E}}^\bullet$.

\begin{lem}
	\label{lem:second-variation-of-second-order-functionals}
	The linearizations of $\dot{\mathcal{E}}^\bullet$ at K\"ahler-Einstein structures are given by
	\begin{align*}
		\ddot{\mathcal{E}}^\Nsym_{ij}
		&=-\frac{1}{8}\tensor{\nabla}{_k}\tensor{\nabla}{^k}\tensor{A}{_i_j}
		-\frac{3}{8}\lambda\tensor{A}{_i_j}
		+\frac{1}{8}\tensor{\nabla}{_[_i}\tensor{\nabla}{^k}\tensor{A}{_j_]_k},\\
		\ddot{\mathcal{E}}^\Nskew_{ij}
		&=-\frac{3}{8}\tensor{\nabla}{_k}\tensor{\nabla}{^k}\tensor{A}{_i_j}
		-\frac{1}{8}\lambda\tensor{A}{_i_j}
		-\frac{5}{8}\tensor{\nabla}{_[_i}\tensor{\nabla}{^k}\tensor{A}{_j_]_k},\\
		\ddot{\mathcal{E}}^T_{ij}
		&=-\tensor{\nabla}{_k}\tensor{\nabla}{^k}\tensor{A}{_i_j},\\
		\ddot{\mathcal{E}}^\tau_{ij}
		&=\tensor{\nabla}{_[_i}\tensor{\nabla}{^k}\tensor{A}{_j_]_k},
	\end{align*}
	where $\Ric(g)=\lambda g$.
\end{lem}

Because of Lemma \ref{lem:second-variation-of-second-order-functionals} and \eqref{eq:linearized-Euler-Lagrange-bis},
the linearized Euler--Lagrange equation of the functional $\mathcal{E}_{(a,b,c,d)}$ equals
$\frac{1}{2}\Delta_{\conj{\partial}}$ when
\begin{equation*}
	-\frac{1}{8}a-\frac{3}{8}b-c=-\frac{1}{2},\qquad
	-\frac{3}{8}a-\frac{1}{8}b=-\frac{1}{2},\qquad
	\frac{1}{8}a-\frac{5}{8}b+d=0.
\end{equation*}
The solutions are $(a,b,c,d)=(1+s,1-3s,s,1/2-2s)$, $s\in\mathbb{R}$;
hence, we obtain Proposition \ref{prop:characterization}.

Among the one-parameter family $\mathcal{E}_{(1+s,1-3s,s,1/2-2s)}$, it seems that there is no special reason
to choose $\mathcal{E}=\mathcal{E}_{(1,1,0,1/2)}$, apart from the simplicity of the expression of the functional.

\bibliography{myrefs}

\end{document}